\documentclass[12pt]{article}
\usepackage{amssymb,amsmath,amsthm}
\usepackage{graphicx}
\usepackage{subcaption}
\usepackage{fullpage}
\usepackage{scrextend}
\usepackage{siunitx}
\usepackage{enumerate}
\usepackage{tikz,calc}
\usepackage{tikz-cd}
\usetikzlibrary{automata,positioning}
\usepackage{subcaption}
\usetikzlibrary{calc}
\usepackage{epsfig,graphicx}
\usepackage{listings}
\usepackage{enumerate}
\usepackage{float}

 \newtheorem{theorem}{Theorem}

\newtheorem{lemma}[theorem]{Lemma}

\newcommand{\ol}{\overline}

\renewcommand{\labelenumi}{(\alph{enumi})}

\title{Combinatorics of intervals in the plane I: trapezoids}
\author{Daniel Di Benedetto\thanks{dibenedetto@math.ubc.ca} \qquad J\'ozsef Solymosi\thanks{solymosi@math.ubc.ca} \qquad Ethan White\thanks{epwhite@math.ubc.ca}\\
Department of Mathematics\\
The University of British Columbia \\
Vancouver, BC\\
 Canada V6T 1Z2}

\begin{document}

\maketitle

\abstract{We study arrangements of intervals in $\mathbb{R}^2$ for which many pairs form trapezoids. We show that any set of intervals forming many trapezoids must have underlying algebraic structure, which we characterise. This leads to some unexpected examples of sets of intervals forming many trapezoids, where an important role is played by degree 2 curves.}

\section{Introduction}

The combinatorics of points in the plane has received much attention, as has the combinatorics of other geometric and algebraic objects such as lines and curves. A typical problem involves estimating the number of possible occurrences of a given combinatorial or geometric relation among all arrangements of the object of study, and beyond this understanding the structure of the extremal arrangements.

Line segments have also been investigated in various areas of discrete and computational geometry, for example in \cite{CE, PS2, T, WS} -- see also \cite{PS1} for more examples -- with research focusing on properties such as intersections and visibility. We will study the combinatorics of some geometric properties of intervals. This is the first in a series of papers on geometric configurations of intervals; in this first part, we consider arrangements for which there are many pairs forming trapezoids.  \\


An \emph{interval} in $\mathbb{R}^2$ is a directed line segment. We will denote an interval by a four-tuple $(a,b;c,d) \in \mathbb{R}^4$, where $(a,b),(c,d) \in \mathbb{R}^2$ denotes the coordinates of the initial and terminal point, respectively. We require that $(a,b) \neq (c,d)$, so the interval has positive length. We call the interval $(c,d;a,b)$ the \emph{reverse} of the interval $(a,b;c,d)$. Consider a pair of distinct intervals, $(a,b;c,d)$ and $(a',b';c',d')$. We say that this pair \emph{forms a trapezoid} if the convex hull of the two intervals is a trapezoid. Arithmetically, this occurs if
\begin{equation}\label{para}(a-a')(d-d')=(b-b')(c-c')  \quad \text{or} \quad (a-c')(d-b')=(c-a')(b-d'),\end{equation}
but note that these equations allow some acceptable degenerate cases, namely where the two intervals lie on the same line, or where they share an endpoint. Note also that both of the above equations can be satisfied simultaneously, as in the case when the two intervals form the diagonals of a parallelogram. All of these scenarios are illustrated in Figure~\ref{pep}.\\

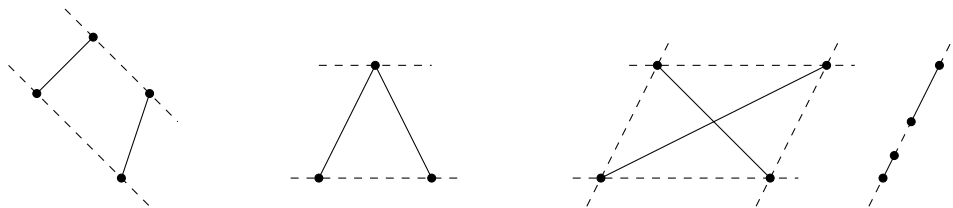
\begin{figure}[h!]
\begin{center}
\begin{tikzpicture}[scale=0.75]

\begin{scope}[yshift = 15mm]
\filldraw

(0,0) circle (2pt)
(1,1) circle (2pt)
(2,0) circle (2pt)
(1.5,-1.5) circle (2pt);

\draw

(0,0)--(1,1)
(1.5,-1.5)--(2,0);

\draw[dashed]

(-0.5,0.5)--(2,-2)
(0.5,1.5)--(2.5,-0.5);

\end{scope}

 \begin{scope}[xshift=50mm]

\filldraw
(0,0) circle (2pt)
(2,0) circle (2pt)
(1,2) circle (2pt);
\draw
(0,0)--(1,2)--(2,0);
\draw[dashed]
(0,2)--(2,2)
(-0.5,0)--(2.5,0);

  \end{scope}
  
   \begin{scope}[xshift=100mm]
   
   \filldraw

(0,0) circle(2pt)
(3,0) circle(2pt)
(4,2) circle(2pt)
(1,2) circle(2pt);

\draw
(0,0)--(4,2)
(1,2)--(3,0);

\draw[dashed]
(-0.25,-0.5)--(1.25,2.5)
(2.75,-0.5)--(4.25,2.5)
(0.5,2)--(4.5,2)
(-0.5,0)--(3.5,0);

    \end{scope}
    
       \begin{scope}[xshift=150mm]
   
   \filldraw

(0,0) circle(2pt)
(0.2,0.4) circle(2pt)
(0.5,1) circle(2pt)
(1,2) circle(2pt);

\draw
(0,0)--(0.2,0.4)
(0.5,1)--(1,2);

\draw[dashed]
(-0.25,-0.5)--(0,0)
(0.2,0.4)--(0.5,1)
(1,2)--(1.25,2.5);
   
    \end{scope}

\end{tikzpicture}
\end{center}
\caption{Intervals forming trapezoids: four cases} \label{pep}
\end{figure}

Given a set of intervals in the plane, a natural combinatorial property of the set is the number of pairs of intervals forming trapezoids. Clearly there are arrangements of intervals for which every pair forms a trapezoid -- the most obvious example is given by placing all endpoints along two parallel lines -- and one can ask if all arrangements with many pairs forming trapezoids must have some shared underlying structure. We answer this question, showing that if the number of pairs forming trapezoids is above a certain threshold, then many of the intervals must be arranged in one of essentially three ways. 

\begin{theorem}\label{maint} Let $\mathfrak{I}$ be a set of $N$ distinct intervals in $\mathbb{R}^2$. If more than $N^{3/2}\log N$ pairs of intervals form trapezoids then one of the following holds. 
\begin{enumerate}[1.]
\item There are two parallel lines in $\mathbb{R}^2$ such that $\gtrsim N^{1/2}$ intervals have an endpoint on each line.
\item There are two parallel lines $\ell_1,\ell_2 \subset \mathbb{R}^2$ such that $\gtrsim N^{1/2}$ intervals $(a,b;c,d) \in \mathfrak{I}$ satisfy $(a,c) \in \ell_1$ and $(b,d) \in \ell_2$.
\item There are two subsets $\mathfrak{I}_1,\mathfrak{I}_2 \subset \mathfrak{I}$ such that for any $i_1 \in \mathfrak{I}_1$ and any $i_2 \in \mathfrak{I}_2$, the intervals $i_1,i_2$ form a trapezoid. In addition, $|\mathfrak{I}_1||\mathfrak{I}_2| \gtrsim N$. 

\end{enumerate}
\end{theorem}

Each of these three cases gives rise to a different geometric structure on the intervals. For examples of each of these, see Figure~\ref{case1}, Figure~\ref{case2} and Figure~\ref{shiftedhyper}, respectively. 

\begin{figure}[h!]
\begin{center}
\begin{tikzpicture}[scale=1]

\filldraw

(0,0) circle (2pt)
(1,0.25) circle (2pt)
(2.5,0.625) circle (2pt)
(2.75,0.6875) circle (2pt)
(4,1) circle (2pt)
(5.5,1.375) circle (2pt)
(6,1.5) circle (2pt)
(6.5,1.625) circle (2pt)
(7.5,1.875) circle (2pt)
( 0.2 , 2.05 ) circle (2pt)
( 1.5 , 2.375 ) circle (2pt)
( 2 , 2.5 ) circle (2pt)
( 4 , 3.0 ) circle (2pt)
( 4.6 , 3.15 ) circle (2pt)
( 5.4 , 3.35 ) circle (2pt)
( 6.7 , 3.675 ) circle (2pt)
( 7 , 3.75 ) circle (2pt)
( 7.3 , 3.825 ) circle (2pt);

\draw
(0,0)--( 4.6 , 3.15 ) 
(2.5,0.625)--( 1.5 , 2.375 )
(4,1) --( 4 , 3.0 )
(1,0.25)--( 0.2 , 2.05 )
(6.5,1.625)--( 2 , 2.5 )
(5.5,1.375)--( 6.7 , 3.675 )
(2.75,0.6875)--( 5.4 , 3.35 )
(6,1.5)--( 7.3 , 3.825 )
(7.5,1.875)--( 7 , 3.75 );

\draw[dashed]
(-0.25,-0.0625)--(8,2)
(-0.25,1.9375)--(8,4);

\end{tikzpicture}
\end{center}
\caption{Case 1} \label{case1}
\end{figure}

\begin{figure}
\begin{center}
\begin{tikzpicture}

\filldraw
(0,0) circle (2pt)
(1,0.5) circle (2pt)
(0.3,2) circle (2pt)
(1.3,2.5) circle (2pt)
(0.7,0.8) circle (2pt)
(1.7,1.3) circle (2pt)
(1.9,1.5) circle (2pt)
(2.9,2) circle (2pt)
(2.6,0.35) circle (2pt)
(3.6,0.85) circle (2pt)
(3.5,1.7) circle (2pt)
(4.5,2.2) circle (2pt)
(4.2,0.3) circle (2pt)
(5.2,0.8) circle (2pt);

\draw
(0,0)--(1,0.5)
(0.3,2)--(1.3,2.5)
(0.7,0.8)--(1.7,1.3)
(1.9,1.5)--(2.9,2)
(2.6,0.35)--(3.6,0.85)
(3.5,1.7)--(4.5,2.2)
(4.2,0.3)--(5.2,0.8);

\draw[dashed]
(2.6,0.35)--(1.9,1.5)
(3.6,0.85)--(2.9,2);

\end{tikzpicture}
\end{center}
\caption{Case 2}
\label{case2}
\end{figure}

\begin{figure}[h!]
\begin{center}
\begin{tikzpicture}[scale=0.8]

\filldraw[blue]

( 2.161209223472559 , 4.763546581352072 )--( 1.682941969615793 , 0.30116867893975674 )
( -1.6645873461885696 , 2.9863011805570787 )--( 1.8185948536513634 , 1.325444263372824 )
( -3.9599699864017817 , 0.30225502291884365 )--( 0.2822400161197344 , 1.1311125046603125 )
( -2.6145744834544478 , -0.8208922323430805 )--( -1.5136049906158564 , -0.10315887444431626 )
( 1.1346487418529052 , 0.6494758216001757 )--( -1.917848549326277 , -1.2425864601263648 )
( 3.8406811466014643 , 3.3615095769028804 )--( -0.5588309963978517 , -1.239585784849292 )
( 3.0156090173732184 , 4.8217777061241875 )--( 1.3139731974375781 , -0.09691565562451554 )
( -0.5820001352344542 , 3.6877164256295365 )--( 1.9787164932467636 , 1.1348582804319953 )
( -3.6445210475387078 , 1.0019764467141594 )--( 0.8242369704835132 , 1.3232487471264336 )
( -3.3562861163058098 , -0.7661852799316446 )--( -1.0880422217787398 , 0.29505041818708255 )
( 0.017702791952203145 , 0.008870982874694677 )--( -1.999980413101407 , -1.0044159045387542 )
( 3.3754158349299686 , 2.6145620814641144 )--( -1.0731458360008699 , -1.380426876732927 )
( 3.6297871258007848 , 4.655227636553674 )--( 0.8403340736532818 , -0.4872797446235553 )
( 0.5469488728313344 , 4.254689147805408 )--( 1.9812147113897407 , 0.8538701374870368 )
( -3.038751651435285 , 1.7811998545965915 )--( 1.3005756803142339 , 1.4099757530159382 )
( -3.8306379212935386 , -0.4911255939768999 )--( -0.5758066333301306 , 0.6697561636583194 )
( -1.1006533522063877 , -0.47312165986230736 )--( -1.9227949837591136 , -0.6862341538279599 )
( 2.6412668329763207 , 1.8186589229448082 )--( -1.501974493543352 , -1.4113039550157562 )
( 3.954818472746677 , 4.277163655699243 )--( 0.2997544193259047 , -0.8388274085237168 )
( 1.6323282472535678 , 4.642054625082039 )--( 1.8258905014552553 , 0.5048631889142356 )
( -2.190917040897074 , 2.577852756623575 )--( 1.673311277072112 , 1.3843848987603244 )
( -3.9998433055785485 , -0.017624271370082045 )--( -0.017702618580807752 , 0.9911095171042332 )
( -2.13133208133359 , -0.7581068490171363 )--( -1.6924408083503413 , -0.3133873838417731 )
( 1.6967160293479877 , 1.0372012906607462 )--( -1.8111567240132478 , -1.3297573693436209 );
\filldraw[red]
( 1.0699953144983494 , 4.46211402808356 )--( -1.927116370834386 , -0.6960593567926057 )
( -2.665104085119296 , 2.1588583817937925 )--( -1.4914104243534405 , -1.4119812334565442 )
( -3.9499190796354595 , -0.29045092810422624 )--( 0.3154913882864964 , -0.8297340757656166 )
( -1.6031966883199018 , -0.6339302176588606 )--( 1.8323318734989098 , 0.5153667646694795 )
( 2.217497344716643 , 1.444213787910519 )--( 1.6645348844478025 , 1.386641778403062 )
( 3.9994345455336604 , 4.033345073735529 )--( -0.033627800968699426 , 0.9830447358990654 )
( 2.1043100695244217 , 4.753028276019339 )--( -1.7008732412571288 , -0.324359103247459 )
( -1.7255073798824831 , 2.941589977571345 )--( -1.8043436675125868 , -1.3335486787269142 )
( -3.968901301810414 , 0.2644581961089165 )--( -0.24890884701412339 , -1.1166797489596652 )
( -2.563305670379973 , -0.8170244547171515 )--( 1.535371619527165 , 0.1268593921685892 )
( 1.1989813731080596 , 0.691452186749852 )--( 1.908038499804178 , 1.253764593179104 )
( 3.858930471546439 , 3.403001653041618 )--( 0.526463582731602 , 1.2279644092524107 )
( 2.970996690814678 , 4.824637869800544 )--( -1.3391395243932047 , 0.07317941050706711 )
( -0.6484577459988731 , 3.64931505554979 )--( -1.9735439285492267 , -1.1488864007743316 )
( -3.6717231216571733 , 0.9576195854326375 )--( -0.7934811462612241 , -1.3146713535449055 )
( -3.319223192282596 , -0.7757161387148566 )--( 1.1161045425735585 , -0.27175352678386977 )
( 0.08495523269458405 , 0.042928754201269737 )--( 1.9995488621460222 , 1.021013239246657 )
( 3.4110262085234924 , 2.6608959250082833 )--( 1.044617179253463 , 1.3750651417576045 )
( 3.601015418989218 , 4.671238430240395 )--( -0.8707307207457864 , 0.46488849437441127 )
( 0.48024766016970694 , 4.225656641756668 )--( -1.9855328116718143 , -0.8727044907934804 )
( -3.082057582634275 , 1.7338164009833403 )--( -1.2748451923004778 , -1.4079369918088076 )
( -3.810733297601074 , -0.5132958664226313 )--( 0.6079292176220942 , -0.6487187155892213 )
( -1.035838392850445 , -0.4496955048973632 )--( 1.9317763084721409 , 0.7069285560234592 )
( 2.691401553273387 , 1.8661436064809067 )--( 1.4795571701557868 , 1.4126289733962403 );

\draw[dashed]
( -3.082057582634275 , 1.7338164009833403 )--( 3.0156090173732184 , 4.8217777061241875 )
( -1.2748451923004778 , -1.4079369918088076 )--( 1.3139731974375781 , -0.09691565562451554 );

\filldraw
( -3.082057582634275 , 1.7338164009833403 ) circle (2.5pt)
( 3.0156090173732184 , 4.8217777061241875 ) circle (2.5pt)
( -1.2748451923004778 , -1.4079369918088076 ) circle (2.5pt)
( 1.3139731974375781 , -0.09691565562451554 ) circle (2.5pt);

\end{tikzpicture}
\end{center}
\caption{Case 3}
\label{shiftedhyper}
\end{figure}

After proving this theorem, we will analyse the geometry of each of the cases. The most interesting case is the third, where we find a class of surprising examples with structure coming from degree 2 curves. Indeed we will see that if the third item holds, then all of the intervals must have endpoints lying on two fixed curves of degree 2. Thus, the underlying structure of a set of intervals forming many trapezoids can be described in the following way.  

\begin{theorem}\label{maint2}
Let $\mathfrak{I}$ be a set of $N$ distinct intervals in $\mathbb{R}^2$. If more than $N^{3/2}\log N$ pairs of intervals form trapezoids then one of the following holds.
\begin{enumerate}[1.]
\item There are two parallel lines in $\mathbb{R}^2$ such that $\gtrsim N^{1/2}$ intervals from $\mathfrak{I}$ have an endpoint on each line.
\item There is a pencil of lines (which may be parallel) such that $\gtrsim N^{1/2}$ intervals from $\mathfrak{I}$ are subsets of lines in the pencil.
\item There are two curves in $\mathbb{R}^2$ of degree $2$ such that $\gtrsim N^{1/2}$ intervals from $\mathfrak{I}$ have an endpoint on each curve.
\end{enumerate}
\end{theorem}

In fact, we will prove some additional conditions about each of these cases. We note that each type of conic (including the degenerate conics) can be present in the third case, and we will give examples for all of them.\\

Throughout we use the notation $A \gtrsim B$ to mean there exists a universal constant $C>0$ for which $A \geq CB$.

\section{Proof of Theorem~\ref{maint}}

The main tool that we will use is a theorem of Guth and Katz on incidences of lines in $\mathbb{R}^3$. The following is a corollary of Theorem~1.2 in \cite{GK} via a standard dyadic summation.

\begin{theorem}\label{tgk}{(Guth--Katz)} Let $\mathfrak{L}$ be a set of $N$ lines in $\mathbb{R}^3$. If the number of pairs of intersecting lines in $\mathfrak{L}$ exceeds $N^{3/2}\log N$, then one of the following holds.
\begin{enumerate}[1.]
    \item There are $\gtrsim N^{1/2}$ concurrent lines in $\mathfrak{L}$.
    \item There exists a plane containing $\gtrsim N^{1/2}$ lines of $\mathfrak{L}$. 
    \item There exists a regulus containing a subset of lines $\mathfrak{L}_R \subset \mathfrak{L}$ such that the number of pairs of intersecting lines in $\mathfrak{L}_R$ is $\gtrsim N$.
\end{enumerate}
\end{theorem}

The reason that reguli and planes appear here is that they are the only doubly ruled surfaces in $\mathbb{R}^3$; the plane is in fact infinitely ruled. Recall that a \emph{doubly ruled surface} is one for which at every point on the surface, there are two distinct lines contained within the surface and containing the point. A regulus contains two families of lines such that there are no intersections within a family but any pair of lines from different families intersect. That is, if we select $M$ lines from one ruling within a regulus and $N$ lines from the other, then there will be precisely $MN$ pairs of intersecting lines.  

We apply this theorem by establishing a correspondence between intervals in the plane and lines in $\mathbb{R}^3$ in such a way that a pair of intervals form a trapezoid if and only if the associated pair of lines intersects. Then pulling back the structures of lines with many intersections given by Theorem~\ref{tgk} leads to the interval structures stated in Theorem~\ref{maint}. This correspondence is obtained in the following lemma. 

\begin{lemma}\label{corresp}
There is a bijection $\mathcal{L}$ from intervals in $\mathbb{R}^2$ to lines in $\mathbb{R}^3$ that are not parallel to the $xy$-plane, such that a pair of intervals forms a trapezoid if and only if their images under $\mathcal{L}$ intersect. 
\end{lemma}

\begin{proof}
To every interval $(a,b;c,d)$ we associate the unique line 
\begin{equation}\label{L} \mathcal{L}(a,b;c,d) = \left \{ \begin{pmatrix} b \\ d \\ 0  \end{pmatrix} + t\begin{pmatrix} a \\ c \\ 1  \end{pmatrix} \colon t \in \mathbb{R} \right\} \subset \mathbb{R}^3.\end{equation}
Conversely, every line $\ell \subset \mathbb{R}^3$ that is not parallel to the $xy$-plane can be normalized to have the same form as the line described above. Define $\mathcal{I}(\ell)$ to be this corresponding interval. Clearly $\mathcal{L}$ is a bijection from intervals in $\mathbb{R}^2$ to lines not parallel to the $xy$-axis in $\mathbb{R}^3$, and $\mathcal{I}$ is its inverse. 

We will now see that under this line--interval correspondence, intervals forming trapezoids correspond to intersecting lines. Let $(a,b;c,d)$ and $(a',b';c',d')$ be intervals such that $(a,b;c,d) \neq (a',b';c',d')$, possibly the reverse of each other. Observe that the lines $\mathcal{L}(a,b;c,d)$ and $\mathcal{L}(a',b';c',d')$ intersect if and only if there is a solution $t \in \mathbb{R}$ to the system of equations 
\begin{equation}\label{LI} t(a-a') = b'-b; \quad t(c-c') = d'-d.\end{equation}
Note that (\ref{LI}) implies the left equation in (\ref{para}). Conversely, the left equation in (\ref{para}) implies there is a unique solution $t \in \mathbb{R}$ to (\ref{LI}), except in the case $a = a'$, $c = c'$, and either $b \neq b'$ or $d \neq d'$. In this exceptional case, the corresponding lines in $\mathbb{R}^3$ are parallel, and so intersect `at infinity' when viewing $\mathbb{R}^3$ embedded inside projective space. Alternatively, for any finite set of intervals one can apply a generic rotation of $\mathbb{R}^2$ so that this exceptional case is avoided and all intersections of the corresponding lines in $\mathbb{R}^3$ take place in affine space.

Note also that the intersection of $\mathcal{L}(a,b;c,d)$ and $\mathcal{L}(a',b';c',d')$ does not imply the right equation in (\ref{para}). Instead, the intersection of $\mathcal{L}(c,d;a,b)$ and $\mathcal{L}(a',b';c',d')$ corresponds to this other pairing of endpoints being parallel. 
\end{proof}

There is a minor technical issue caused by the fact that intervals are directed, whereas we want to count trapezoids formed by either direction. The following lemma shows how we will deal with this. 

\begin{lemma}\label{i2l} Let $\mathfrak{I}$ be a set of $N$ distinct intervals in $\mathbb{R}^2$ and let $T$ be the set of trapezoids formed by pairs of intervals in $\mathfrak{I}$. Then there is a set $\mathfrak{L}(\mathfrak{I})$ of $2N$ lines in $\mathbb{R}^3$ such that
\[ 2 |T|= \# \{\text{Pairs of intersecting lines in }\mathfrak{L}(\mathfrak{I})\}-N.\]
Note that in the above counting, if two intervals $(a,b;c,d), (a',b';c'd') \in \mathfrak{I}$ satisfy both equations of (\ref{para}) then we count the trapezoid with multiplicity two. 
\end{lemma}

\begin{proof}
Let $\mathfrak{I} = \{(a_i,b_i;c_i,d_i) \colon 1 \leq i \leq N \}$ be a set of $N$ distinct intervals in $\mathbb{R}^2$. Associate with $\mathfrak{I}$ the set 
\[ \mathfrak{L}(\mathfrak{I}) = \{ \mathcal{L}(a_i,b_i;c_i,d_i) \colon 1 \leq i \leq N \}\cup \{ \mathcal{L}(c_i,d_i;a_i,b_i) \colon 1 \leq i \leq N \} ,\]
of $2N$ lines in $\mathbb{R}^3$, where $\mathcal{L}$ is the map from Lemma \ref{corresp}. This is the set of lines associated to either an interval in $\mathfrak{I}$, or its reverse direction. As mentioned previously, $\mathbb{R}^2$ can be rotated so that all resulting lines in $\mathfrak{L}$ have intersections in affine space, not at infinity. Henceforth, we will assume $\mathfrak{I}$ has this property. 

For every pair of intervals $(a,b;c,d),(a',b',c',d') \in \mathfrak{I}$ satisfying one (resp. two) equation(s) in (\ref{para}), there are two (resp. four) intersections among the lines
\[ \mathcal{L}(a,b;c,d), \ \mathcal{L}(c,d;a,b), \ \mathcal{L}(a',b';c',d'), \ \mathcal{L}(c',d';a',b') .\]
Note that $\mathcal{L}(a_i,b_i;c_i,d_i)$ and $\mathcal{L}(c_i,d_i;a_i,b_i)$ intersect exactly once for all $1 \leq i \leq N$. These intersections correspond to an interval forming a trapezoid with its reverse direction. All other pairs of intersecting lines in $\mathfrak{L}$ correspond to a trapezoid formed by distinct line segments. 
\end{proof}

Having set up this correspondence, we can now prove Theorem~\ref{maint} by applying Theorem~\ref{tgk}.  

\begin{proof}[Proof of Theoerem~\ref{maint}] Let $\mathfrak{I}$ be a set of $N$ distinct intervals, and let $\mathfrak{L} = \mathfrak{L}(\mathfrak{I})$ be as above. By Lemma~\ref{i2l}, if at least $N^{3/2}\log N$ pairs of intervals form trapezoids, then more than $N^{3/2}\log N$ pairs of lines intersect in $\mathfrak{L}$, excluding the interval-reverse interval intersections. By Theorem~\ref{tgk}, many lines are concurrent, lie in a plane, or are contained in a regulus. We consider these three cases separately, which correspond to 1, 2, and 3 above, respectively. \\

\emph{Case 1: Concurrent lines.} Suppose that $\gtrsim N^{1/2}$ lines of $\mathfrak{L}$ pass through the point $(u,v,w)\in \mathbb{R}^3$. Then $\gtrsim N^{1/2}$ lines of $\mathfrak{L}$ are of the form $[u-aw,v-cw,0]+t[a,c,1]$. These lines correspond to $\gtrsim N^{1/2}$ intervals with one endpoint on $y = u-wx$ and the other on $y = v-wx$. Note that if $u = v$, then many intervals are contained entirely on the line $y = u-wx$. In this case, an interval and its reverse might be represented in the set of concurrent lines passing through $(u,v,w)$. \\ 

\emph{Case 2: Lines in a plane.} Suppose that $\gtrsim N^{1/2}$ lines lie in the plane $Ax+By+Cz+D=0$. A line in $\mathfrak{L}$ of the form (\ref{L}) belongs to this plane if 
\[ t(Aa+Bc+C) + Ab+Bd+D = 0,\]
for all $t \in \mathbb{R}$. Thus $(a,c)$ is on the line $Ax+By+C=0$ and $(b,d)$ is on the line $Ax+By+D=0$. Note that $A,B$ are not both zero, since no line in $\mathfrak{L}$ is parallel to the $xy$-plane. We conclude that $\gtrsim N^{1/2}$ intervals $(a,b;c,d)$ of $\mathfrak{I}$ satisfy $Aa+Bc+C=Ab+Bd+D=0$.\\ 

\emph{Case 3: Lines in a regulus.} Suppose there exists a subset $\mathfrak{L}_R \subset \mathfrak{L}$ such that the number of pairs of intersecting lines in $\mathfrak{L}_R$ is $\gtrsim N$. As noted earlier, $\mathfrak{L}_R$ can be partitioned into $\mathfrak{L}_1$ and $\mathfrak{L}_2$, where the lines in $\mathfrak{L}_1$ belong to one ruling of the regulus and $\mathfrak{L}_2$ from the other. There are no intersecting pairs of lines within the same ruling, thus $|\mathfrak{L}_1||\mathfrak{L}_2| \gtrsim N$. Pulling back $\mathfrak{L}_1,\mathfrak{L}_2$ to intervals in $\mathbb{R}^2$ gives $\mathfrak{I}_1,\mathfrak{I}_2$ as described. 
\end{proof}


\section{Analysis of underlying geometry}

In this section we further discuss the geometry underlying the situations 1, 2, 3 of Theorem~\ref{maint}, proving Theorem~\ref{maint2} before giving examples.\\

\textbf{Case 1 (Concurrent):} The geometry in this case is clear. For example, see Figure~\ref{case1}. Any pair of intervals with endpoints on the two lines forms a trapezoid, resulting in $\gtrsim N$ trapezoids. We can see that the bounds in Theorem~\ref{maint} are almost tight via the following situation. Given $\gtrsim N^{1/2}$ points in $\mathbb{R}^3$ each with $\gtrsim N^{1/2}$ lines of $\mathfrak{L}$ passing through them, the pullback to the plane creates $\gtrsim N^{1/2}$ pairs of parallel lines in $\mathbb{R}^2$, with $\gtrsim N^{1/2}$ intervals between each pair. This totals $\gtrsim N^{3/2}$ pairs of intervals forming trapezoids.   \\

\textbf{Case 2 (Coplanar):} Suppose that many intervals $(a,b;c,d)$ of $\mathfrak{I}$ satisfy $Aa+Bc+C=Ab+Bd+D=0$. If $A+B \neq 0$, we have $\frac{A}{A+B}(a,b) + \frac{B}{A+B}(c,d) = (-\frac{C}{A+B},-\frac{D}{A+B})$. In other words, the line containing the points $(a,b)$ and $(c,d)$ always contains $(-\frac{C}{A+B},-\frac{D}{A+B})$. Moreover, the ratio of the distance between $(a,b)$ and $(-\frac{C}{A+B},-\frac{D}{A+B})$ to $(c,d)$ and $(-\frac{C}{A+B},-\frac{D}{A+B})$ is $B:A$. If $AB>0$ then the point $(-\frac{C}{A+B},-\frac{D}{A+B})$ is on $(a,b;c,d)$, and if $AB=0$ it is an endpoint of each interval. A special case of this is $A=B$ and pairs of intervals are the diagonals of a parallelogram. On the other hand, if $A+B=0$, then $a+C/A = c$ and $b+D/A = d$. This corresponds to a set of intervals that are translates of each other. Hence, in this case the intervals are contained on a pencil of lines. Note that the additional condition on the ratio of the distances means that there is a homothety such that the terminal points $(c,d)$ are the images of the initial points $(a,b)$ under this transformation. See Figure~\ref{case2eg} for examples.\\

\begin{figure}[h!]
\begin{center}
\begin{tikzpicture}[scale=0.8]

\begin{scope}[xshift=5mm,yshift=0mm]

\filldraw
( 2.6775 , 0.0 ) circle (2.5 pt)
( 4.7925 , 5.4975000000000005 ) circle (2.5 pt)
( 3.6975000000000007 , 0.66 ) circle (2.5 pt)
( 4.425000000000001 , 2.9625 ) circle (2.5 pt)
( 0.585 , 3.93 ) circle (2.5 pt)
( 0.0 , 2.0925000000000002 ) circle (2.5 pt)
( 2.6400000000000006 , 5.197500000000001 ) circle (2.5 pt)
( 1.8675000000000002 , 2.6774999999999998 ) circle (2.5 pt)
( 0.7049999999999996 , 1.8324999999999996 ) circle (2.5 pt)
( 0.0 , 0.0 ) circle (2.5 pt)
( 0.36499999999999977 , 1.6124999999999998 ) circle (2.5 pt)
( 0.12249999999999983 , 0.8449999999999999 ) circle (2.5 pt)
( 1.4024999999999999 , 0.5225 ) circle (2.5 pt)
( 1.5974999999999997 , 1.1349999999999998 ) circle (2.5 pt)
( 0.7174999999999996 , 0.09999999999999998 ) circle (2.5 pt)
( 0.9749999999999996 , 0.94 ) circle (2.5 pt);
\draw
( 2.6775 , 0.0 )--( 0.7049999999999996 , 1.8324999999999996 )
( 4.7925 , 5.4975000000000005 )--( 0.0 , 0.0 )
( 3.6975000000000007 , 0.66 )--( 0.36499999999999977 , 1.6124999999999998 )
( 4.425000000000001 , 2.9625 )--( 0.12249999999999983 , 0.8449999999999999 )
( 0.585 , 3.93 )--( 1.4024999999999999 , 0.5225 )
( 0.0 , 2.0925000000000002 )--( 1.5974999999999997 , 1.1349999999999998 )
( 2.6400000000000006 , 5.197500000000001 )--( 0.7174999999999996 , 0.09999999999999998 )
( 1.8675000000000002 , 2.6774999999999998 )--( 0.9749999999999996 , 0.94 );
   \draw
(1.5,-1) node{$A=1$, $B=3$} ;

\draw[dashed]
( 0.585 , 3.93 )--( 4.7925 , 5.4975000000000005 )
( 1.4024999999999999 , 0.5225 )--( 0.0 , 0.0 );
   
\end{scope}

   \begin{scope}[xshift=100mm]
   \draw
(1.5,-1) node{$A=-1$, $B=3$} ;
   
   \filldraw
( 0.0 , 0.51 ) circle (2.5 pt)
( 3.7314999999999996 , 0.799 ) circle (2.5 pt)
( 4.675 , 0.765 ) circle (2.5 pt)
( 1.87 , 4.5305 ) circle (2.5 pt)
( 0.9520000000000001 , 0.0 ) circle (2.5 pt)
( 3.5614999999999997 , 1.5130000000000001 ) circle (2.5 pt)
( 0.2209999999999998 , 3.0345 ) circle (2.5 pt)
( 3.8165 , 4.488 ) circle (2.5 pt)
( 0.0 , 0.16999999999999996 ) circle (2.5 pt)
( 1.2438333333333333 , 0.26633333333333337 ) circle (2.5 pt)
( 1.5583333333333333 , 0.255 ) circle (2.5 pt)
( 0.6233333333333333 , 1.5101666666666664 ) circle (2.5 pt)
( 0.3173333333333334 , 0.0 ) circle (2.5 pt)
( 1.1871666666666665 , 0.5043333333333334 ) circle (2.5 pt)
( 0.07366666666666662 , 1.0114999999999998 ) circle (2.5 pt)
( 1.2721666666666667 , 1.4959999999999998 ) circle (2.5 pt);
\draw
( 0.0 , 0.51 )--( 0.0 , 0.16999999999999996 )
( 3.7314999999999996 , 0.799 )--( 1.2438333333333333 , 0.26633333333333337 )
( 4.675 , 0.765 )--( 1.5583333333333333 , 0.255 )
( 1.87 , 4.5305 )--( 0.6233333333333333 , 1.5101666666666664 )
( 0.9520000000000001 , 0.0 )--( 0.3173333333333334 , 0.0 )
( 3.5614999999999997 , 1.5130000000000001 )--( 1.1871666666666665 , 0.5043333333333334 )
( 0.2209999999999998 , 3.0345 )--( 0.07366666666666662 , 1.0114999999999998 )
( 3.8165 , 4.488 )--( 1.2721666666666667 , 1.4959999999999998 );

\draw[dashed]

( 3.8165 , 4.488 )--( 3.5614999999999997 , 1.5130000000000001 )
( 1.2721666666666667 , 1.4959999999999998 )--( 1.1871666666666665 , 0.5043333333333334 );

   \end{scope}

      \begin{scope}[xshift=5mm,yshift=-60mm]
      \draw
(1.5,-1) node{$A=0$, $B=1$} ;
   \filldraw
( 2.0199999999999996 , 2.93 ) circle (2.5 pt)
( 0.72 , 1.25 ) circle (2.5 pt)
( 3.0700000000000003 , 0.4 ) circle (2.5 pt)
( 0.0 , 0.0 ) circle (2.5 pt)
( 0.19000000000000006 , 1.4600000000000002 ) circle (2.5 pt)
( 5.35 , 3.66 ) circle (2.5 pt)
( 4.93 , 0.16 ) circle (2.5 pt)
( 5.319999999999999 , 4.21 ) circle (2.5 pt)
( 0.0 , 0.0 ) circle (2.5 pt)
( 0.0 , 0.0 ) circle (2.5 pt)
( 0.0 , 0.0 ) circle (2.5 pt)
( 0.0 , 0.0 ) circle (2.5 pt)
( 0.0 , 0.0 ) circle (2.5 pt)
( 0.0 , 0.0 ) circle (2.5 pt)
( 0.0 , 0.0 ) circle (2.5 pt)
( 0.0 , 0.0 ) circle (2.5 pt);
\draw
( 2.0199999999999996 , 2.93 )--( 0.0 , 0.0 )
( 0.72 , 1.25 )--( 0.0 , 0.0 )
( 3.0700000000000003 , 0.4 )--( 0.0 , 0.0 )
( 0.0 , 0.0 )--( 0.0 , 0.0 )
( 0.19000000000000006 , 1.4600000000000002 )--( 0.0 , 0.0 )
( 5.35 , 3.66 )--( 0.0 , 0.0 )
( 4.93 , 0.16 )--( 0.0 , 0.0 )
( 5.319999999999999 , 4.21 )--( 0.0 , 0.0 );

\draw[dashed]
( 2.0199999999999996 , 2.93 )--( 5.319999999999999 , 4.21 );

   \end{scope}

      \begin{scope}[xshift=100mm,yshift=-60mm]
      \draw
(1.5,-1) node{$A=1$, $B=-1$} ;
   \filldraw
( 2.685 , 2.3600000000000003 ) circle (2.5 pt)
( 1.955 , 3.0120000000000005 ) circle (2.5 pt)
( 3.01 , 0.508 ) circle (2.5 pt)
( 3.73 , 2.064 ) circle (2.5 pt)
( 0.46 , 0.27599999999999997 ) circle (2.5 pt)
( 0.145 , 2.2920000000000003 ) circle (2.5 pt)
( 0.06 , 0.44400000000000006 ) circle (2.5 pt)
( 2.445 , 1.936 ) circle (2.5 pt)
( 3.685 , 2.3600000000000003 ) circle (2.5 pt)
( 2.955 , 3.0120000000000005 ) circle (2.5 pt)
( 4.01 , 0.508 ) circle (2.5 pt)
( 4.73 , 2.064 ) circle (2.5 pt)
( 1.46 , 0.27599999999999997 ) circle (2.5 pt)
( 1.145 , 2.2920000000000003 ) circle (2.5 pt)
( 1.06 , 0.44400000000000006 ) circle (2.5 pt)
( 3.445 , 1.936 ) circle (2.5 pt);
\draw
( 2.685 , 2.3600000000000003 )--( 3.685 , 2.3600000000000003 )
( 1.955 , 3.0120000000000005 )--( 2.955 , 3.0120000000000005 )
( 3.01 , 0.508 )--( 4.01 , 0.508 )
( 3.73 , 2.064 )--( 4.73 , 2.064 )
( 0.46 , 0.27599999999999997 )--( 1.46 , 0.27599999999999997 )
( 0.145 , 2.2920000000000003 )--( 1.145 , 2.2920000000000003 )
( 0.06 , 0.44400000000000006 )--( 1.06 , 0.44400000000000006 )
( 2.445 , 1.936 )--( 3.445 , 1.936 );

\draw[dashed]
( 0.06 , 0.44400000000000006 )--( 2.445 , 1.936 )
( 1.06 , 0.44400000000000006 )--( 3.445 , 1.936 );

   \end{scope}

\end{tikzpicture}
\end{center}
\caption{Intervals coming from lines contained in a plane} \label{case2eg}
\end{figure}
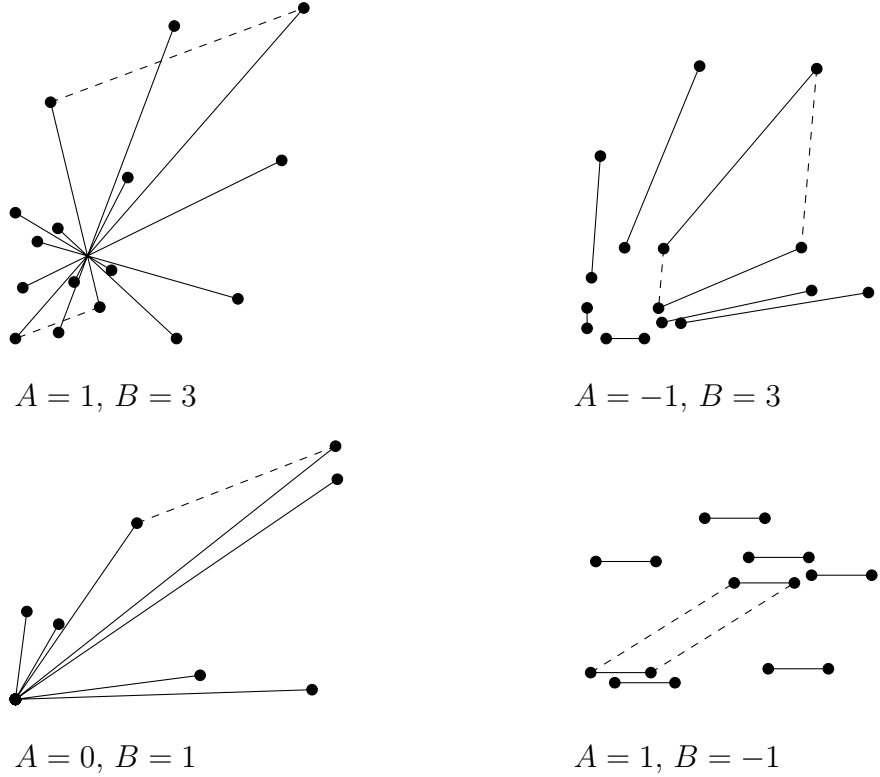


\textbf{Case 3 (Regulus):} As was mentioned previously, the lines in a regulus have a complete bipartite structure with respect to intersections. This therefore reveals an interesting class of examples of sets of intervals forming trapezoids in a complete bipartite way, and many of these examples seem difficult to discover without pulling intervals back from reguli in $\mathbb{R}^3$. We begin with an overview of the general structure of these intervals, followed by a closer examination of specific cases. 

Let $\mathfrak{I}_1,\mathfrak{I}_2,\mathfrak{L}_1,\mathfrak{L}_2$ be as in the proof of Theorem~\ref{maint}. Let $\ell_j = \{[b_j,d_j,0]+t[a_j,c_j,1] \colon t \in \mathbb{R}\}$ for $j = 1,2,3$, be lines in $\mathfrak{L}_2$. A line $\ell = \{[b,d,0]+t[a,c,1] \colon t \in \mathbb{R}\} \in \mathfrak{L}_1$ intersects each $\ell_j$, $j=1,2,3$ and therefore satisfies 
\begin{equation}\label{parsys}
    (a-a_j)(d-d_j) = (b-b_j)(c-c_j), \quad j= 1,2,3.
\end{equation} 
The above is a system of three degree two polynomial equations. Subtracting two equations of the system gives the two linear equations
\begin{align}\label{linsys}
    (d_1-d_2)a+(c_2-c_1)b+(b_2-b_1)c+(a_1-a_2)d &= b_2c_2-b_1c_1 \nonumber \\
    (d_2-d_3)a+(c_3-c_2)b+(b_3-b_2)c+(a_2-a_3)d &= b_3c_3-b_2c_2.
\end{align}

We consider three encompassing cases in the system (\ref{linsys}). 

\emph{Subcase (i):} The $2 \times 2$ coefficient matrix induced by the coefficients of $a$ and $b$ in (\ref{linsys}) is invertible, or the $2 \times 2$ coefficient matrix induced by the coefficients of $c$ and $d$ in (\ref{linsys}) is invertible. If the $2 \times 2$ matrix induced by the coefficients of $a$ and $b$ in (\ref{linsys}) is invertible, then Gaussian elimination on (\ref{linsys}) gives $a$ and $b$ linearly in terms of $c$ and $d$. Substituting these linear relations into one equation of (\ref{parsys}) gives a degree two polynomial in $c$ and $d$, i.e. the set of all $(c,d)$ lie on a conic or line. Thus the set of all $(a,b)$ also lies on a conic or line. An analogous result follows if the $2 \times 2$ matrix induced by the coefficients of $c$ and $d$ in (\ref{linsys}) is invertible. See Figures~\ref{case3a} and~\ref{case3b} for examples. 

\emph{Subcase (ii):} Neither $2 \times 2$ coefficient matrix of Subcase (i) is invertible, but (\ref{linsys}) has rank 2. Then all solutions $(a,b)$ lie on one line, and solutions $(c,d)$ on a line, i.e. $b$ can be given linearly in $a$, and $d$ linearly in $c$. Substituting these linear relations into (\ref{parsys}) shows the set of all $(a,d)$ lie on a conic. See Figure~\ref{case3lines} for example.

\emph{Subcase (iii):} (\ref{linsys}) has rank 1. In this case, the points of all of the sets $\{(a_j,b_j)\}_{j=1}^3$, $\{(a_j,c_j)\}_{j=1}^3$, and $\{(a_j,d_j)\}_{j=1}^3$ are colinear. It follows that there exist $m_1,m_2,m_3,r_1,r_2,r_3 \in \mathbb{R}$ such that $(a_j,b_j,c_j,d_j) = (a_j,m_1a_j+r_1,m_2a_j+r_2,m_3a_j+r_3)$ for $j=1,2,3$. Substituting this relation into (\ref{parsys}) gives a quadratic equation in $a_j^2$ with a $a_j^2$ coefficient of $m_3-m_1m_2$. Since this quadratic equation has at least three solutions (namely $a_1,a_2,a_3$), the leading coefficient is zero, i.e. $m_3 - m_1m_2 = 0$. This relation implies $\ell_1,\ell_2,\ell_3$ all pass through the point $(r_1,-r_2m_1+r_3,-m_1)$. This cannot happen, since lines of the same ruling in a regulus do not intersect. We conclude that Subcase (iii) never occurs. The same argument also shows that in Subcase (i), the solution set of $(a,b)$ lies on a conic, not a line. \\

The preceding analysis shows that the endpoints of $\mathfrak{I}_1$ and $\mathfrak{I}_2$ lie on a conic (Subcase (i)), or a line (Subcase (ii)), thus completing the proof of Theorem~\ref{maint2}. We will now showcase examples involving each type of conic, as well as the Subcase (ii) situation. To facilitate simpler computation, we examine axis parallel reguli. Up to a rigid motion (rotation and translation) of $\mathbb{R}^3$, any regulus takes the form 
\begin{equation}\label{hyper}
\frac{x^2}{A^2}+\frac{y^2}{B^2}-\frac{z^2}{C^2}=1, 
\end{equation}
or
\begin{equation}\label{parab}
z= \frac{x^2}{A^2}-\frac{y^2}{B^2}.
\end{equation}
Equation (\ref{hyper}) describes a hyperboloid of one sheet, and (\ref{parab}) describes a hyperbolic paraboloid. We determine what the corresponding interval set looks like in each case, and then discuss the arrangements that result from rigid transformations of these standard forms of reguli.

The hyperboloid (\ref{hyper}) is ruled by the following two families of lines, parameterized by $\theta \in [0,2\pi]$.
\[ \begin{pmatrix} A \sin \theta \\ -B\cos \theta \\ 0 \end{pmatrix} + t \begin{pmatrix} \frac{A}{C}\cos \theta \\ \frac{B}{C} \sin \theta \\ 1 \end{pmatrix}, \quad \text{and} \quad \begin{pmatrix} A \sin \theta \\ B\cos \theta \\ 0 \end{pmatrix} + t \begin{pmatrix} \frac{A}{C}\cos \theta \\ -\frac{B}{C} \sin \theta \\ 1 \end{pmatrix}, t \in \mathbb{R}. \]
This corresponds to the following two families of intervals parameterized by $\theta \in [0,2\pi]$.
\[ \left( \frac{A}{C}\cos \theta, A \sin \theta ; \frac{B}{C} \sin \theta, -B\cos \theta \right), \quad \text{and} \quad \left( \frac{A}{C}\cos \theta, A \sin \theta ; -\frac{B}{C} \sin \theta, B\cos \theta \right).\]
Observe that the above intervals have their first endpoints on the ellipse $Cx^2+y^2=A^2$ and their second on $Cx^2+y^2=B^2$. Furthermore, the second endpoint in the first family has a phase shift of $-\pi/2$ radians compared to the first endpoint. In the second family, the second endpoint has a phase shift of $+\pi/2$ compared to the first endpoint. If $A=B$ and $C=1$, then the two families of intervals are the reverse of each other, hence any pair of intervals from a union of intervals from the two families forms a trapezoid. In all other cases, the two sets of intervals have bipartite structure. Examples of both of these cases are shown in Figure~\ref{case3a}. \\

\begin{figure}[h!]
\begin{center}
\begin{tikzpicture}[scale=0.75]

\begin{scope}[xshift=0mm]

\draw
(0,-4) node{$A=B=C=1$} ;

\draw[red]
( 2.5 , 0.0 )--( 0.0 , -2.5 )
( 2.3776412907378837 , 0.7725424859373685 )--( 0.7725424859373685 , -2.3776412907378837 )
( 2.0225424859373686 , 1.469463130731183 )--( 1.469463130731183 , -2.0225424859373686 )
( 1.4694631307311832 , 2.022542485937368 )--( 2.022542485937368 , -1.4694631307311832 )
( 0.7725424859373686 , 2.3776412907378837 )--( 2.3776412907378837 , -0.7725424859373686 )
( 1.5308084989341916e-16 , 2.5 )--( 2.5 , -1.5308084989341916e-16 )
( -0.7725424859373684 , 2.377641290737884 )--( 2.377641290737884 , 0.7725424859373684 )
( -1.4694631307311825 , 2.0225424859373686 )--( 2.0225424859373686 , 1.4694631307311825 )
( -2.022542485937368 , 1.4694631307311832 )--( 1.4694631307311832 , 2.022542485937368 )
( -2.3776412907378837 , 0.7725424859373687 )--( 0.7725424859373687 , 2.3776412907378837 )
( -2.5 , 3.061616997868383e-16 )--( 3.061616997868383e-16 , 2.5 )
( -2.377641290737884 , -0.7725424859373682 )--( -0.7725424859373682 , 2.377641290737884 )
( -2.0225424859373686 , -1.4694631307311825 )--( -1.4694631307311825 , 2.0225424859373686 )
( -1.4694631307311832 , -2.022542485937368 )--( -2.022542485937368 , 1.4694631307311832 )
( -0.7725424859373689 , -2.3776412907378837 )--( -2.3776412907378837 , 0.7725424859373689 )
( -4.592425496802574e-16 , -2.5 )--( -2.5 , 4.592425496802574e-16 )
( 0.7725424859373681 , -2.377641290737884 )--( -2.377641290737884 , -0.7725424859373681 )
( 1.4694631307311823 , -2.022542485937369 )--( -2.022542485937369 , -1.4694631307311823 )
( 2.022542485937368 , -1.4694631307311832 )--( -1.4694631307311832 , -2.022542485937368 )
( 2.3776412907378837 , -0.7725424859373691 )--( -0.7725424859373691 , -2.3776412907378837 );
\draw[blue]
( 2.5 , 0.0 )--( -0.0 , 2.5 )
( 2.3776412907378837 , 0.7725424859373685 )--( -0.7725424859373685 , 2.3776412907378837 )
( 2.0225424859373686 , 1.469463130731183 )--( -1.469463130731183 , 2.0225424859373686 )
( 1.4694631307311832 , 2.022542485937368 )--( -2.022542485937368 , 1.4694631307311832 )
( 0.7725424859373686 , 2.3776412907378837 )--( -2.3776412907378837 , 0.7725424859373686 )
( 1.5308084989341916e-16 , 2.5 )--( -2.5 , 1.5308084989341916e-16 )
( -0.7725424859373684 , 2.377641290737884 )--( -2.377641290737884 , -0.7725424859373684 )
( -1.4694631307311825 , 2.0225424859373686 )--( -2.0225424859373686 , -1.4694631307311825 )
( -2.022542485937368 , 1.4694631307311832 )--( -1.4694631307311832 , -2.022542485937368 )
( -2.3776412907378837 , 0.7725424859373687 )--( -0.7725424859373687 , -2.3776412907378837 )
( -2.5 , 3.061616997868383e-16 )--( -3.061616997868383e-16 , -2.5 )
( -2.377641290737884 , -0.7725424859373682 )--( 0.7725424859373682 , -2.377641290737884 )
( -2.0225424859373686 , -1.4694631307311825 )--( 1.4694631307311825 , -2.0225424859373686 )
( -1.4694631307311832 , -2.022542485937368 )--( 2.022542485937368 , -1.4694631307311832 )
( -0.7725424859373689 , -2.3776412907378837 )--( 2.3776412907378837 , -0.7725424859373689 )
( -4.592425496802574e-16 , -2.5 )--( 2.5 , -4.592425496802574e-16 )
( 0.7725424859373681 , -2.377641290737884 )--( 2.377641290737884 , 0.7725424859373681 )
( 1.4694631307311823 , -2.022542485937369 )--( 2.022542485937369 , 1.4694631307311823 )
( 2.022542485937368 , -1.4694631307311832 )--( 1.4694631307311832 , 2.022542485937368 )
( 2.3776412907378837 , -0.7725424859373691 )--( 0.7725424859373691 , 2.3776412907378837 );

\draw[dashed]
( 1.4694631307311832 , 2.022542485937368 )--( -2.5 , 4.592425496802574e-16 )
( 2.022542485937368 , -1.4694631307311832 )--( -4.592425496802574e-16 , -2.5 );

\filldraw
( 1.4694631307311832 , 2.022542485937368 ) circle (8/3 pt)
( -2.5 , 4.592425496802574e-16 ) circle (8/3 pt)
( 2.022542485937368 , -1.4694631307311832 ) circle (8/3 pt)
( -4.592425496802574e-16 , -2.5 ) circle (8/3 pt);

\end{scope}

\begin{scope}[xshift=70mm]

\draw
(0,-4) node{$A=3$, $B=2$, $C=1$} ;
\draw[red]
( 3.0 , 0.0 )--( 0.0 , -2.0 )
( 2.8531695488854605 , 0.9270509831248421 )--( 0.6180339887498948 , -1.902113032590307 )
( 2.4270509831248424 , 1.7633557568774194 )--( 1.1755705045849463 , -1.618033988749895 )
( 1.7633557568774196 , 2.427050983124842 )--( 1.6180339887498947 , -1.1755705045849465 )
( 0.9270509831248424 , 2.8531695488854605 )--( 1.902113032590307 , -0.6180339887498949 )
( 1.8369701987210297e-16 , 3.0 )--( 2.0 , -1.2246467991473532e-16 )
( -0.927050983124842 , 2.853169548885461 )--( 1.9021130325903073 , 0.6180339887498947 )
( -1.7633557568774192 , 2.4270509831248424 )--( 1.618033988749895 , 1.175570504584946 )
( -2.427050983124842 , 1.7633557568774196 )--( 1.1755705045849465 , 1.6180339887498947 )
( -2.8531695488854605 , 0.9270509831248426 )--( 0.618033988749895 , 1.902113032590307 )
( -3.0 , 3.6739403974420594e-16 )--( 2.4492935982947064e-16 , 2.0 )
( -2.853169548885461 , -0.9270509831248419 )--( -0.6180339887498946 , 1.9021130325903073 )
( -2.4270509831248424 , -1.7633557568774192 )--( -1.175570504584946 , 1.618033988749895 )
( -1.7633557568774196 , -2.427050983124842 )--( -1.6180339887498947 , 1.1755705045849465 )
( -0.9270509831248427 , -2.8531695488854605 )--( -1.902113032590307 , 0.6180339887498951 )
( -5.51091059616309e-16 , -3.0 )--( -2.0 , 3.6739403974420594e-16 )
( 0.9270509831248417 , -2.853169548885461 )--( -1.9021130325903073 , -0.6180339887498945 )
( 1.7633557568774187 , -2.427050983124843 )--( -1.6180339887498951 , -1.1755705045849458 )
( 2.427050983124842 , -1.7633557568774196 )--( -1.1755705045849465 , -1.6180339887498947 )
( 2.8531695488854605 , -0.9270509831248428 )--( -0.6180339887498952 , -1.902113032590307 );
\draw[blue]
( 3.0 , 0.0 )--( -0.0 , 2.0 )
( 2.8531695488854605 , 0.9270509831248421 )--( -0.6180339887498948 , 1.902113032590307 )
( 2.4270509831248424 , 1.7633557568774194 )--( -1.1755705045849463 , 1.618033988749895 )
( 1.7633557568774196 , 2.427050983124842 )--( -1.6180339887498947 , 1.1755705045849465 )
( 0.9270509831248424 , 2.8531695488854605 )--( -1.902113032590307 , 0.6180339887498949 )
( 1.8369701987210297e-16 , 3.0 )--( -2.0 , 1.2246467991473532e-16 )
( -0.927050983124842 , 2.853169548885461 )--( -1.9021130325903073 , -0.6180339887498947 )
( -1.7633557568774192 , 2.4270509831248424 )--( -1.618033988749895 , -1.175570504584946 )
( -2.427050983124842 , 1.7633557568774196 )--( -1.1755705045849465 , -1.6180339887498947 )
( -2.8531695488854605 , 0.9270509831248426 )--( -0.618033988749895 , -1.902113032590307 )
( -3.0 , 3.6739403974420594e-16 )--( -2.4492935982947064e-16 , -2.0 )
( -2.853169548885461 , -0.9270509831248419 )--( 0.6180339887498946 , -1.9021130325903073 )
( -2.4270509831248424 , -1.7633557568774192 )--( 1.175570504584946 , -1.618033988749895 )
( -1.7633557568774196 , -2.427050983124842 )--( 1.6180339887498947 , -1.1755705045849465 )
( -0.9270509831248427 , -2.8531695488854605 )--( 1.902113032590307 , -0.6180339887498951 )
( -5.51091059616309e-16 , -3.0 )--( 2.0 , -3.6739403974420594e-16 )
( 0.9270509831248417 , -2.853169548885461 )--( 1.9021130325903073 , 0.6180339887498945 )
( 1.7633557568774187 , -2.427050983124843 )--( 1.6180339887498951 , 1.1755705045849458 )
( 2.427050983124842 , -1.7633557568774196 )--( 1.1755705045849465 , 1.6180339887498947 )
( 2.8531695488854605 , -0.9270509831248428 )--( 0.6180339887498952 , 1.902113032590307 );

\draw[dashed]
( -2.4270509831248424 , -1.7633557568774192 )--( -2.853169548885461 , -0.9270509831248419 )
( 1.175570504584946 , -1.618033988749895 )--( -0.6180339887498946 , 1.9021130325903073 );

\filldraw
( -2.4270509831248424 , -1.7633557568774192 ) circle (8/3 pt)
( -2.853169548885461 , -0.9270509831248419 ) circle (8/3 pt)
( 1.175570504584946 , -1.618033988749895 ) circle (8/3 pt)
( -0.6180339887498946 , 1.9021130325903073 ) circle (8/3 pt);

\end{scope}

\begin{scope}[xshift=140mm]

\draw
(0,-4) node{$A=1$, $B=3$, $C=2$} ;
\draw[red]
( 0.5 , 0.0 )--( 0.0 , -3.0 )
( 0.47552825814757677 , 0.3090169943749474 )--( 0.46352549156242107 , -2.8531695488854605 )
( 0.4045084971874737 , 0.5877852522924731 )--( 0.8816778784387097 , -2.4270509831248424 )
( 0.2938926261462366 , 0.8090169943749473 )--( 1.213525491562421 , -1.7633557568774196 )
( 0.15450849718747373 , 0.9510565162951535 )--( 1.4265847744427302 , -0.9270509831248424 )
( 3.061616997868383e-17 , 1.0 )--( 1.5 , -1.8369701987210297e-16 )
( -0.15450849718747367 , 0.9510565162951536 )--( 1.4265847744427305 , 0.927050983124842 )
( -0.2938926261462365 , 0.8090169943749475 )--( 1.2135254915624212 , 1.7633557568774192 )
( -0.40450849718747367 , 0.5877852522924732 )--( 0.8816778784387098 , 2.427050983124842 )
( -0.47552825814757677 , 0.3090169943749475 )--( 0.4635254915624213 , 2.8531695488854605 )
( -0.5 , 1.2246467991473532e-16 )--( 1.8369701987210297e-16 , 3.0 )
( -0.4755282581475768 , -0.3090169943749473 )--( -0.46352549156242095 , 2.853169548885461 )
( -0.4045084971874737 , -0.587785252292473 )--( -0.8816778784387096 , 2.4270509831248424 )
( -0.2938926261462366 , -0.8090169943749473 )--( -1.213525491562421 , 1.7633557568774196 )
( -0.15450849718747378 , -0.9510565162951535 )--( -1.4265847744427302 , 0.9270509831248427 )
( -9.184850993605148e-17 , -1.0 )--( -1.5 , 5.51091059616309e-16 )
( 0.15450849718747361 , -0.9510565162951536 )--( -1.4265847744427305 , -0.9270509831248417 )
( 0.29389262614623646 , -0.8090169943749476 )--( -1.2135254915624214 , -1.7633557568774187 )
( 0.40450849718747367 , -0.5877852522924732 )--( -0.8816778784387098 , -2.427050983124842 )
( 0.47552825814757677 , -0.3090169943749476 )--( -0.4635254915624214 , -2.8531695488854605 );
\draw[blue]
( 0.5 , 0.0 )--( -0.0 , 3.0 )
( 0.47552825814757677 , 0.3090169943749474 )--( -0.46352549156242107 , 2.8531695488854605 )
( 0.4045084971874737 , 0.5877852522924731 )--( -0.8816778784387097 , 2.4270509831248424 )
( 0.2938926261462366 , 0.8090169943749473 )--( -1.213525491562421 , 1.7633557568774196 )
( 0.15450849718747373 , 0.9510565162951535 )--( -1.4265847744427302 , 0.9270509831248424 )
( 3.061616997868383e-17 , 1.0 )--( -1.5 , 1.8369701987210297e-16 )
( -0.15450849718747367 , 0.9510565162951536 )--( -1.4265847744427305 , -0.927050983124842 )
( -0.2938926261462365 , 0.8090169943749475 )--( -1.2135254915624212 , -1.7633557568774192 )
( -0.40450849718747367 , 0.5877852522924732 )--( -0.8816778784387098 , -2.427050983124842 )
( -0.47552825814757677 , 0.3090169943749475 )--( -0.4635254915624213 , -2.8531695488854605 )
( -0.5 , 1.2246467991473532e-16 )--( -1.8369701987210297e-16 , -3.0 )
( -0.4755282581475768 , -0.3090169943749473 )--( 0.46352549156242095 , -2.853169548885461 )
( -0.4045084971874737 , -0.587785252292473 )--( 0.8816778784387096 , -2.4270509831248424 )
( -0.2938926261462366 , -0.8090169943749473 )--( 1.213525491562421 , -1.7633557568774196 )
( -0.15450849718747378 , -0.9510565162951535 )--( 1.4265847744427302 , -0.9270509831248427 )
( -9.184850993605148e-17 , -1.0 )--( 1.5 , -5.51091059616309e-16 )
( 0.15450849718747361 , -0.9510565162951536 )--( 1.4265847744427305 , 0.9270509831248417 )
( 0.29389262614623646 , -0.8090169943749476 )--( 1.2135254915624214 , 1.7633557568774187 )
( 0.40450849718747367 , -0.5877852522924732 )--( 0.8816778784387098 , 2.427050983124842 )
( 0.47552825814757677 , -0.3090169943749476 )--( 0.4635254915624214 , 2.8531695488854605 );

\draw[dashed]
( 0.47552825814757677 , -0.3090169943749476 )--( -0.15450849718747378 , -0.9510565162951535 )
( 0.4635254915624214 , 2.8531695488854605 )--( -1.4265847744427302 , 0.9270509831248427 );

\filldraw
( 0.47552825814757677 , -0.3090169943749476 ) circle (8/3 pt)
( -0.15450849718747378 , -0.9510565162951535 ) circle (8/3 pt)
( 0.4635254915624214 , 2.8531695488854605 ) circle (8/3 pt)
( -1.4265847744427302 , 0.9270509831248427 ) circle (8/3 pt);

\end{scope}

\end{tikzpicture}
\end{center}
\caption{Intervals from the hyperboloid $x^2/A^2+y^2/B^2-z^2/C^2=1$.} \label{case3a}
\end{figure}
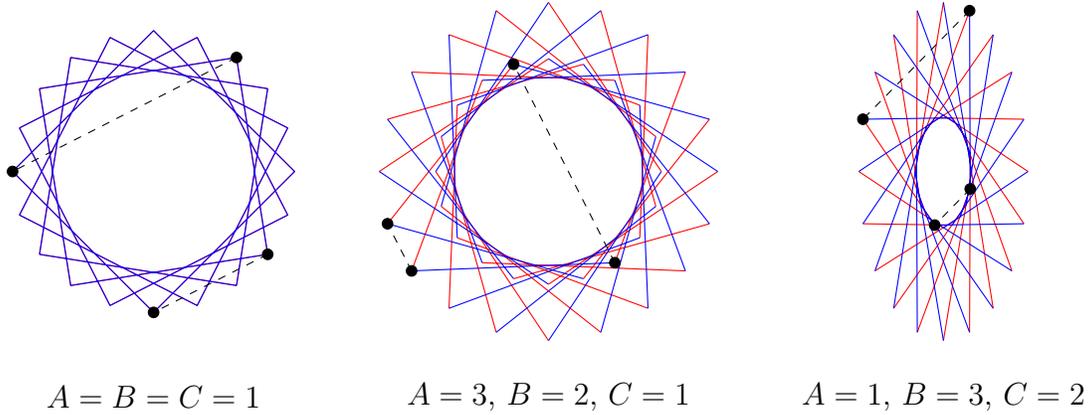

The hyperbolic paraboloid (\ref{parab}) is ruled by the following two families of lines, parameterized by $\lambda \in \mathbb{R}\setminus \{0\}$. 
\[ \begin{pmatrix} A/(2\lambda)\\ -B/(2\lambda)\\ 0 \end{pmatrix} + t \begin{pmatrix} A\lambda/2 \\ B\lambda/2 \\ 1 \end{pmatrix},\quad \text{and} \quad \begin{pmatrix} A/(2\lambda)\\ B/(2\lambda)\\ 0 \end{pmatrix} + t \begin{pmatrix} A\lambda/2 \\ -B\lambda/2 \\ 1 \end{pmatrix}, t \in \mathbb{R}.\]
This corresponds to the following two families of intervals parameterized by $\lambda \in \mathbb{R}\setminus \{0\}$. 
\[ (A\lambda/2,A/(2\lambda);B\lambda/2,-B/(2\lambda)) \quad \text{and} \quad (A\lambda/2,A/(2\lambda);-B\lambda/2,B/(2\lambda)) .\]
The left endpoint of all the above intervals lie on the hyperbola $y = \frac{A^2}{4x}$ and the right endpoint lies on $y = -\frac{B^2}{4x}$. The intervals in the first family intersect the $x$-axis, and the intervals in the second family intersect the $y$-axis. See Figure~\ref{case3b} for an example of intervals coming from hyperbolic paraboloids.

\begin{figure}[H]
\begin{center}
\begin{tikzpicture}[scale=1.4]

\draw[red]
( 0.25 , 1.0 )--( 0.25 , -1.0 )
( 0.5 , 0.5 )--( 0.5 , -0.5 )
( 0.75 , 0.3333333333333333 )--( 0.75 , -0.3333333333333333 )
( 1.0 , 0.25 )--( 1.0 , -0.25 )
( 1.25 , 0.2 )--( 1.25 , -0.2 )
( 1.5 , 0.16666666666666666 )--( 1.5 , -0.16666666666666666 )
( 1.75 , 0.14285714285714285 )--( 1.75 , -0.14285714285714285 )
( 2.0 , 0.125 )--( 2.0 , -0.125 )
( 2.25 , 0.1111111111111111 )--( 2.25 , -0.1111111111111111 )
( 2.5 , 0.1 )--( 2.5 , -0.1 )
( -0.25 , -1.0 )--( -0.25 , 1.0 )
( -0.5 , -0.5 )--( -0.5 , 0.5 )
( -0.75 , -0.3333333333333333 )--( -0.75 , 0.3333333333333333 )
( -1.0 , -0.25 )--( -1.0 , 0.25 )
( -1.25 , -0.2 )--( -1.25 , 0.2 )
( -1.5 , -0.16666666666666666 )--( -1.5 , 0.16666666666666666 )
( -1.75 , -0.14285714285714285 )--( -1.75 , 0.14285714285714285 )
( -2.0 , -0.125 )--( -2.0 , 0.125 )
( -2.25 , -0.1111111111111111 )--( -2.25 , 0.1111111111111111 )
( -2.5 , -0.1 )--( -2.5 , 0.1 );
\draw[blue]
( 2.5 , 0.1 )--( -2.5 , 0.1 )
( 1.25 , 0.2 )--( -1.25 , 0.2 )
( 0.8333333333333334 , 0.3 )--( -0.8333333333333334 , 0.3 )
( 0.625 , 0.4 )--( -0.625 , 0.4 )
( 0.5 , 0.5 )--( -0.5 , 0.5 )
( 0.4166666666666667 , 0.6 )--( -0.4166666666666667 , 0.6 )
( 0.35714285714285715 , 0.7 )--( -0.35714285714285715 , 0.7 )
( 0.3125 , 0.8 )--( -0.3125 , 0.8 )
( 0.2777777777777778 , 0.8999999999999999 )--( -0.2777777777777778 , 0.8999999999999999 )
( 0.25 , 1.0 )--( -0.25 , 1.0 )
( -2.5 , -0.1 )--( 2.5 , -0.1 )
( -1.25 , -0.2 )--( 1.25 , -0.2 )
( -0.8333333333333334 , -0.3 )--( 0.8333333333333334 , -0.3 )
( -0.625 , -0.4 )--( 0.625 , -0.4 )
( -0.5 , -0.5 )--( 0.5 , -0.5 )
( -0.4166666666666667 , -0.6 )--( 0.4166666666666667 , -0.6 )
( -0.35714285714285715 , -0.7 )--( 0.35714285714285715 , -0.7 )
( -0.3125 , -0.8 )--( 0.3125 , -0.8 )
( -0.2777777777777778 , -0.8999999999999999 )--( 0.2777777777777778 , -0.8999999999999999 )
( -0.25 , -1.0 )--( 0.25 , -1.0 );

\draw
(0,-1.6) node{$A=1$, $B=2$} ;

\draw[dashed]
( 1.0 , 0.25 )--( 0.3125 , 0.8 )
( 1.0 , -0.25 )--( -0.3125 , 0.8 );

\filldraw
( 1.0 , 0.25 ) circle (2/1.4 pt)
( 0.3125 , 0.8 ) circle (2/1.4 pt)
( 1.0 , -0.25 ) circle (2/1.4 pt)
( -0.3125 , 0.8 ) circle (2/1.4 pt);

\begin{scope}[xshift=60mm]

\draw[red]
( 0.15 , 0.6 )--( 0.3 , -1.2 )
( 0.3 , 0.3 )--( 0.6 , -0.6 )
( 0.44999999999999996 , 0.19999999999999998 )--( 0.8999999999999999 , -0.39999999999999997 )
( 0.6 , 0.15 )--( 1.2 , -0.3 )
( 0.75 , 0.12 )--( 1.5 , -0.24 )
( 0.8999999999999999 , 0.09999999999999999 )--( 1.7999999999999998 , -0.19999999999999998 )
( 1.05 , 0.08571428571428572 )--( 2.1 , -0.17142857142857143 )
( 1.2 , 0.075 )--( 2.4 , -0.15 )
( 1.3499999999999999 , 0.06666666666666667 )--( 2.6999999999999997 , -0.13333333333333333 )
( 1.5 , 0.06 )--( 3.0 , -0.12 )
( -0.15 , -0.6 )--( -0.3 , 1.2 )
( -0.3 , -0.3 )--( -0.6 , 0.6 )
( -0.44999999999999996 , -0.19999999999999998 )--( -0.8999999999999999 , 0.39999999999999997 )
( -0.6 , -0.15 )--( -1.2 , 0.3 )
( -0.75 , -0.12 )--( -1.5 , 0.24 )
( -0.8999999999999999 , -0.09999999999999999 )--( -1.7999999999999998 , 0.19999999999999998 )
( -1.05 , -0.08571428571428572 )--( -2.1 , 0.17142857142857143 )
( -1.2 , -0.075 )--( -2.4 , 0.15 )
( -1.3499999999999999 , -0.06666666666666667 )--( -2.6999999999999997 , 0.13333333333333333 )
( -1.5 , -0.06 )--( -3.0 , 0.12 );
\draw[blue]
( 1.5 , 0.06 )--( -3.0 , 0.12 )
( 0.75 , 0.12 )--( -1.5 , 0.24 )
( 0.5 , 0.18 )--( -1.0 , 0.36 )
( 0.375 , 0.24 )--( -0.75 , 0.48 )
( 0.3 , 0.3 )--( -0.6 , 0.6 )
( 0.25 , 0.36 )--( -0.5 , 0.72 )
( 0.21428571428571427 , 0.42 )--( -0.42857142857142855 , 0.84 )
( 0.1875 , 0.48 )--( -0.375 , 0.96 )
( 0.16666666666666666 , 0.5399999999999999 )--( -0.3333333333333333 , 1.0799999999999998 )
( 0.15 , 0.6 )--( -0.3 , 1.2 )
( -1.5 , -0.06 )--( 3.0 , -0.12 )
( -0.75 , -0.12 )--( 1.5 , -0.24 )
( -0.5 , -0.18 )--( 1.0 , -0.36 )
( -0.375 , -0.24 )--( 0.75 , -0.48 )
( -0.3 , -0.3 )--( 0.6 , -0.6 )
( -0.25 , -0.36 )--( 0.5 , -0.72 )
( -0.21428571428571427 , -0.42 )--( 0.42857142857142855 , -0.84 )
( -0.1875 , -0.48 )--( 0.375 , -0.96 )
( -0.16666666666666666 , -0.5399999999999999 )--( 0.3333333333333333 , -1.0799999999999998 )
( -0.15 , -0.6 )--( 0.3 , -1.2 );

\draw
(0,-1.6) node{$A=1$, $B=2$} ;

\draw[dashed]
( -1.05 , -0.08571428571428572 )--( 0.375 , 0.24 )
( -2.1 , 0.17142857142857143 )--( -0.75 , 0.48 );

\filldraw
( -1.05 , -0.08571428571428572 ) circle (2/1.4 pt)
( 0.375 , 0.24 ) circle (2/1.4 pt)
( -2.1 , 0.17142857142857143 ) circle (2/1.4 pt)
( -0.75 , 0.48 ) circle (2/1.4 pt);

\end{scope}

\end{tikzpicture}
\end{center}
\caption{Intervals from the hyperbolic paraboloid $z=x^2/A^2-y^2/B^2$}
\label{case3b}
\end{figure}
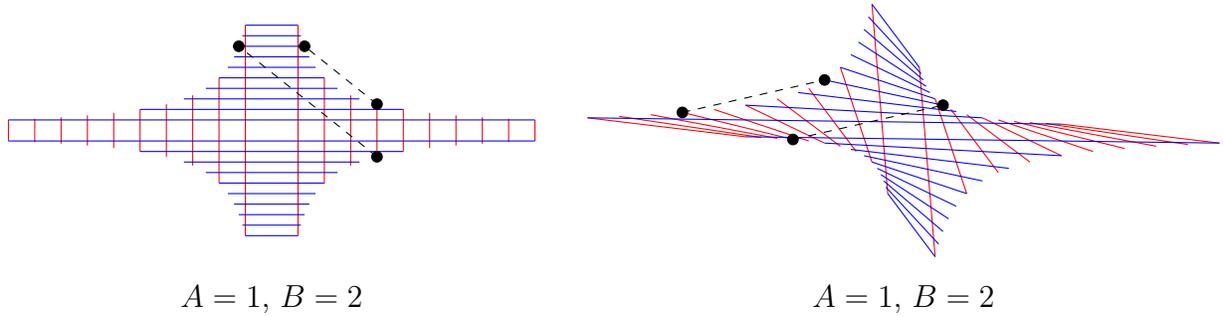


It is easy to understand the effect of translations in $\mathbb{R}^3$ on sets of intervals in the plane, and by this we can completely describe the sets of intervals corresponding to axis parallel reguli in $\mathbb{R}^3$. A translation by the vector $(p , q, r) \in \mathbb{R}^3$ maps the line $(b,d,0)+t(a,c,1)$ to $(b+p,c+q,r)+t(a,c,1)$. Hence the corresponding transformation on intervals maps $(a,b;c,d)$ to $(a,b+p-ra;c,d+q-rc)$. The effect of the translation by $(p,q,r)$ in $\mathbb{R}^3$ on $\mathfrak{I}$ can therefore be described by the composition of three basic maps. First, a vertical shift of all left endpoints of intervals in $\mathfrak{I}$ by $p$. Second, a vertical shift of all right endpoints of intervals in $\mathfrak{I}$ by $q$. Third, an affine shear transformation acting on all of $\mathbb{R}^2$ by a factor $r$. See Figure~\ref{shiftedhyper} for an example of a set of intervals resulting from a translation of a hyperboloid, given by translating the hyperboloid of the form (\ref{hyper}) with $A=2,B=1,C=1/2$ by $(2,0,1/2)$.

Evidently, $\mathbb{R}^3$ translations do not change the type of conic that the endpoints lie on, so it is now clear that axis parallel reguli only produce intervals with endpoints lying on either ellipses or hyperbolas. In what follows, we will see that by rotating the axis parallel reguli it is also possible to produce parabolas and pairs of lines (which is a degenerate conic), thereby showing that each type of conic can be realised in this way. \\

If an interval is considered as a point in $\mathbb{R}^4$, then $\mathbb{R}^3$ translations induce an affine transformation on $\mathbb{R}^4$. The effect of $\mathbb{R}^3$ rotations on $\mathfrak{I}$ is more complicated: the map induced by rotations of $\mathbb{R}^3$ is in general not affine. Rotations in $\mathbb{R}^3$ are described by well-known matrices. For example, calculation with such a matrix shows that rotation around the $x$-axis by an angle $\alpha$ induces the map
\[ (a,b;c,d) \mapsto \left(\frac{a}{c\sin \alpha + \cos \alpha }, \frac{(bc-ad)\sin \alpha + b \cos \alpha}{c \sin \alpha + \cos \alpha}; \frac{c\cos \alpha - \sin \alpha}{c \sin \alpha + \cos \alpha}, \frac{d}{c \sin \alpha + \cos \alpha} \right) ,\]
on intervals in $\mathbb{R}^2$. Thus, the type of conic containing the endpoints of the intervals is in general not preserved under these transformations, and indeed there are examples where parabolas and pairs of lines arise as a result of rotating the axis parallel reguli -- see Figure~\ref{3e}.\\

\begin{figure}[h!]
\begin{center}
\begin{tikzpicture}[scale=0.25]

\begin{scope}[xshift=-300mm, yshift = 80mm]

\filldraw[blue]

( 1 , -0.75 )--( 1.5 , -0.5625 )
( 1 , -1.5 )--( 3.0 , -2.25 )
( 1 , -2.25 )--( 4.5 , -5.0625 )
( 1 , -3.0 )--( 6.0 , -9.0 )
( 1 , -3.75 )--( 7.5 , -14.0625 )
( 1 , -4.5 )--( 9.0 , -20.25 )
( 1 , 0.0 )--( -0.0 , -0.0 )
( 1 , 0.75 )--( -1.5 , -0.5625 )
( 1 , 1.5 )--( -3.0 , -2.25 )
( 1 , 2.25 )--( -4.5 , -5.0625 )
( 1 , 3.0 )--( -6.0 , -9.0 )
( 1 , 3.75 )--( -7.5 , -14.0625 )
( 1 , 4.5 )--( -9.0 , -20.25 );

\filldraw[red]

( -1 , -0.75 )--( -1.5 , -0.5625 )
( -1 , -1.5 )--( -3.0 , -2.25 )
( -1 , -2.25 )--( -4.5 , -5.0625 )
( -1 , -3.0 )--( -6.0 , -9.0 )
( -1 , -3.75 )--( -7.5 , -14.0625 )
( -1 , -4.5 )--( -9.0 , -20.25 )
( -1 , 0.0 )--( 0.0 , -0.0 )
( -1 , 0.75 )--( 1.5 , -0.5625 )
( -1 , 1.5 )--( 3.0 , -2.25 )
( -1 , 2.25 )--( 4.5 , -5.0625 )
( -1 , 3.0 )--( 6.0 , -9.0 )
( -1 , 3.75 )--( 7.5 , -14.0625 )
( -1 , 4.5 )--( 9.0 , -20.25 );

\draw[dashed]
( -1 , -3.75 )--( 1 , -1.5 )
( -7.5 , -14.0625 )--( 3.0 , -2.25 );

\filldraw
( -1 , -3.75 ) circle (8pt)
( 1 , -1.5 ) circle (8pt)
( -7.5 , -14.0625 ) circle (8pt)
( 3.0 , -2.25 ) circle (8pt);

\end{scope}

\filldraw[blue]

( -0.35355339059327373 , -0.13258252147247762 )--( 1.7677669529663687 , -0.9280776503073436 )
( -1.414213562373095 , 0.5303300858899107 )--( 2.82842712474619 , -2.6516504294495533 )
( -2.4748737341529163 , 1.9887378220871645 )--( 3.8890872965260113 , -5.170718337426628 )
( -3.5355339059327373 , 4.242640687119286 )--( 4.949747468305832 , -8.48528137423857 )
( -4.596194077712559 , 7.29203868098627 )--( 6.010407640085653 , -12.595339539885376 )
( 0.7071067811865475 , 0.0 )--( 0.7071067811865475 , 0.0 )
( 1.7677669529663687 , 0.9280776503073436 )--( -0.35355339059327373 , 0.13258252147247762 )
( 2.82842712474619 , 2.6516504294495533 )--( -1.414213562373095 , -0.5303300858899107 )
( 3.8890872965260113 , 5.170718337426628 )--( -2.4748737341529163 , -1.9887378220871645 )
( 4.949747468305832 , 8.48528137423857 )--( -3.5355339059327373 , -4.242640687119286 )
( 6.010407640085653 , 12.595339539885376 )--( -4.596194077712559 , -7.29203868098627 );
\filldraw[red]
( -0.7071067811865475 , 0.0 )--( -0.7071067811865475 , -0.0 )
( 0.35355339059327373 , -0.13258252147247762 )--( -1.7677669529663687 , -0.9280776503073436 )
( 1.414213562373095 , 0.5303300858899107 )--( -2.82842712474619 , -2.6516504294495533 )
( 2.4748737341529163 , 1.9887378220871645 )--( -3.8890872965260113 , -5.170718337426628 )
( 3.5355339059327373 , 4.242640687119286 )--( -4.949747468305832 , -8.48528137423857 )
( 4.596194077712559 , 7.29203868098627 )--( -6.010407640085653 , -12.595339539885376 )
( -0.7071067811865475 , 0.0 )--( -0.7071067811865475 , 0.0 )
( -1.7677669529663687 , 0.9280776503073436 )--( 0.35355339059327373 , 0.13258252147247762 )
( -2.82842712474619 , 2.6516504294495533 )--( 1.414213562373095 , -0.5303300858899107 )
( -3.8890872965260113 , 5.170718337426628 )--( 2.4748737341529163 , -1.9887378220871645 )
( -4.949747468305832 , 8.48528137423857 )--( 3.5355339059327373 , -4.242640687119286 )
( -6.010407640085653 , 12.595339539885376 )--( 4.596194077712559 , -7.29203868098627 );

\draw[dashed]
( -4.949747468305832 , 8.48528137423857 )--( 3.8890872965260113 , 5.170718337426628 )
( 3.5355339059327373 , -4.242640687119286 )--( -2.4748737341529163 , -1.9887378220871645 );

\filldraw
( -4.949747468305832 , 8.48528137423857 ) circle (8pt)
( 3.8890872965260113 , 5.170718337426628 ) circle (8pt)
( 3.5355339059327373 , -4.242640687119286 ) circle (8pt)
( -2.4748737341529163 , -1.9887378220871645 ) circle (8pt);

\end{tikzpicture}
\end{center}
\caption{Rotating the hyperbolic paraboloid $z=x^2-y^2$ by $\pi/2$ around the $x$-axis produces intervals with endpoints on the lines $x=\pm1$ and the parabola $y=-\frac{x^2}{4}$; applying a further rotation by $\pi/4$ around the $z$-axis gives intervals with endpoints on the parabolas $y=\pm\frac{2x^2-1}{4\sqrt{2}}$}
\label{3e}
\end{figure}
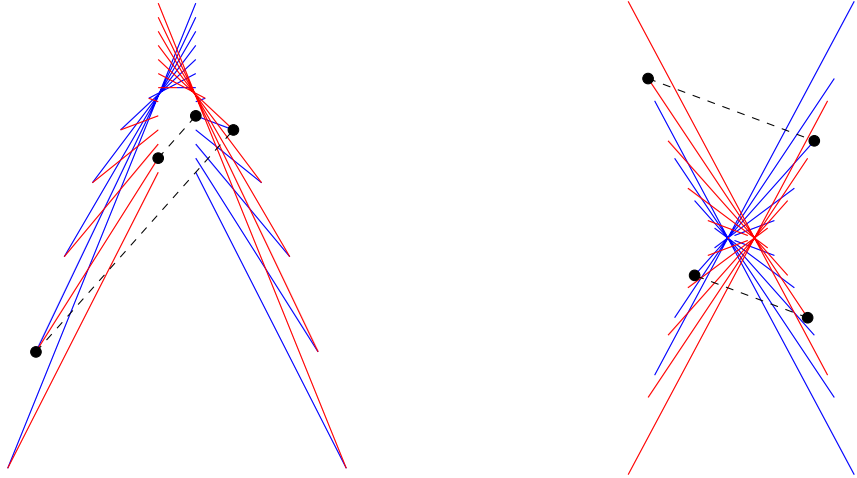

All the Case 3 examples that we have seen so far belong to Subcase (i). We finish this section with an example from Subcase (ii). Recall that we want a family of intervals $\{(a_i,b_i;c_i,d_i) \}_i$ such that $(a_i,b_i)$ and $(c_i,d_i)$ lie on lines, and $(a_i,c_i)$ lie on a conic, for all $i$. By rotating and scaling, we assume that $(a_i,b_i)$ lie on the line $y=x$ and $(a_i,c_i)$ lie on the hyperbola $y=1/x$. These choices determine the following two families of intervals 
\begin{equation*}
   \{ (t,t;1/t,u/t+v)  \colon t \in \mathbb{R}\} \quad \text{and} \quad  \{ (t,ut;1/t,v+1/t) \colon t \in \mathbb{R} \} ,
\end{equation*}
where $u,v \in \mathbb{R}$, and $u \neq 1$. These intervals correspond to lines belonging to the hyperboloid 
\begin{equation}\label{linereg}
xy = z^2+z(u+1)+u+vx.
\end{equation}
When $u=1$, the two families are identical, and (\ref{linereg}) is a cone. See Figure~\ref{case3lines} for a drawing of this case.

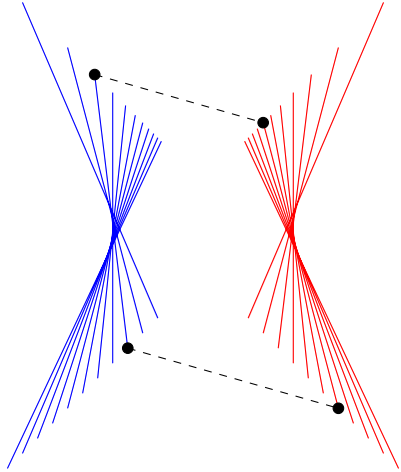
\begin{figure}[h!]
\begin{center}
\begin{tikzpicture}[scale=1.2]

\filldraw[blue]

( -0.5 , -0.5 )--( -2.0 , 3.0 )
( -0.6666666666666666 , -0.6666666666666666 )--( -1.5 , 2.5 )
( -0.8333333333333334 , -0.8333333333333334 )--( -1.2 , 2.2 )
( -1.0 , -1.0 )--( -1.0 , 2.0 )
( -1.1666666666666667 , -1.1666666666666667 )--( -0.8571428571428571 , 1.8571428571428572 )
( -1.3333333333333333 , -1.3333333333333333 )--( -0.75 , 1.75 )
( -1.5 , -1.5 )--( -0.6666666666666666 , 1.6666666666666665 )
( -1.6666666666666667 , -1.6666666666666667 )--( -0.6 , 1.6 )
( -1.8333333333333333 , -1.8333333333333333 )--( -0.5454545454545455 , 1.5454545454545454 )
( -2.0 , -2.0 )--( -0.5 , 1.5 )
( -2.1666666666666665 , -2.1666666666666665 )--( -0.46153846153846156 , 1.4615384615384617 );
\filldraw[red]
( 0.5 , -0.5 )--( 2.0 , 3.0 )
( 0.6666666666666666 , -0.6666666666666666 )--( 1.5 , 2.5 )
( 0.8333333333333334 , -0.8333333333333334 )--( 1.2 , 2.2 )
( 1.0 , -1.0 )--( 1.0 , 2.0 )
( 1.1666666666666667 , -1.1666666666666667 )--( 0.8571428571428571 , 1.8571428571428572 )
( 1.3333333333333333 , -1.3333333333333333 )--( 0.75 , 1.75 )
( 1.5 , -1.5 )--( 0.6666666666666666 , 1.6666666666666665 )
( 1.6666666666666667 , -1.6666666666666667 )--( 0.6 , 1.6 )
( 1.8333333333333333 , -1.8333333333333333 )--( 0.5454545454545455 , 1.5454545454545454 )
( 2.0 , -2.0 )--( 0.5 , 1.5 )
( 2.1666666666666665 , -2.1666666666666665 )--( 0.46153846153846156 , 1.4615384615384617 );

\draw[dashed]
( -0.8333333333333334 , -0.8333333333333334 )--( 1.5 , -1.5 )
( -1.2 , 2.2 )--( 0.6666666666666666 , 1.6666666666666665 );

\filldraw
( -0.8333333333333334 , -0.8333333333333334 ) circle (5/3 pt)
( 1.5 , -1.5 ) circle (5/3 pt)
( -1.2 , 2.2 ) circle (5/3 pt)
( 0.6666666666666666 , 1.6666666666666665 ) circle (5/3 pt);

\end{tikzpicture}
\end{center}
\caption{The hyperboloid (\ref{linereg}) with $u=-1,v=1$, produces intervals with endpoints on two pairs of lines}
\label{case3lines}
\end{figure}

\section{Orthodiagonal quadrilaterals}

An orthodiagonal quadrilateral is a convex quadrilateral with perpendicular diagonals. Several geometric and arithmetic characterizations of orthodiagonal quadrilaterals are known. For example, a convex quadrilateral is orthodiagonal if and only if the midpoints of the sides are the vertices of a rectangle. Another well known characterization is that the sum of the lengths of opposite sides is equal -- see for example \cite{J} and the references contained therein.

Our proof of Theorem~\ref{maint} can easily be modified to deal with some variants of the the problem we have considered. An example is sets of intervals for which there are many pairs \emph{forming orthodiagonal quadrilaterals}, meaning that the convex hull of the two intervals has perpendicular diagonals. This property is illustrated in the leftmost diagram of Figure~\ref{perpint}. Arithmetically, two intervals $(a,b;c,d)$, $(a',b';c','d)$ forming an orthodiagonal quadrilateral satisfy 
\begin{equation}\label{perpprop}(b-b')(d-d')=-(a-a')(c-c') \quad \text{or} \quad (b-d')(d-b')=-(a-c')(c-a'). \end{equation}
The arithmetic conditions (\ref{perpprop}) are not exclusive to orthodiagonal quadrilaterals, i.e. other pairs of intervals can satisfy one or both of (\ref{perpprop}), and we illustrate such possibilities in Figure~\ref{perpint}.

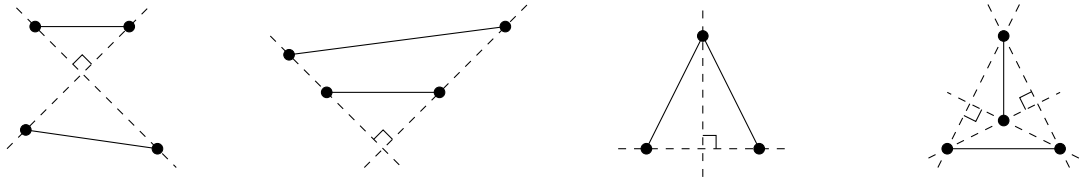
\begin{figure}[h!]
\begin{center}
\begin{tikzpicture}[scale=0.5]

\begin{scope}[yshift = 2cm]

\draw

(-1.25,1.25)--(1.25,1.25)
(-1.5,-1.5)--(2,-2);

\filldraw

(2,-2) circle (4pt)
(-1.5,-1.5) circle (4pt)
(-1.25,1.25) circle (4pt)
(1.25,1.25) circle (4pt);

\draw

(-0.25,0.25)--(0,0.5)--(0.25,0.25);

\draw[dashed]

(-1.75,1.75)--(2.5,-2.5)
(-2,-2)--(1.75,1.75);

\end{scope}
\begin{scope}[xshift= 80mm]

\draw

(-2.5,2.5)--(3.25,3.25)
(-1.5,1.5)--(1.5,1.5);

\filldraw

(-2.5,2.5) circle (4pt)
(3.25,3.25) circle (4pt)
(-1.5,1.5) circle (4pt)
(1.5,1.5) circle (4pt);

\draw

(-0.25,0.25)--(0,0.5)--(0.25,0.25);

\draw[dashed]

(-3,3)--(0.5,-0.5)
(-0.5,-0.5)--(4,4);

\end{scope}

 \begin{scope}[xshift=150mm]

\filldraw
(0,0) circle (4pt)
(3,0) circle (4pt)
(1.5,3) circle (4pt);
\draw
(0,0)--(1.5,3)--(3,0);
\draw[dashed]
(1.5,-0.75)--(1.5,3.75)
(-0.75,0)--(3.75,0);

\draw

(1.5,0.3535533905932738)--(1.8535533905932738,0.3535533905932738)--(1.8535533905932738,0);

  \end{scope}

 \begin{scope}[xshift=230mm]
 
\def\a{3}
\def\m{2}
\def\n{0.5}

\filldraw
(0,0) circle (4pt)
(\a,0) circle (4pt)
({\a/(\m*\n+1)},{\a*\m/(\m*\n+1)}) circle (4pt)
({\a/(\m*\n+1)},{\a*\n/(\m*\n+1)}) circle (4pt);
\draw 
(0,0)--(\a,0)
({\a/(\m*\n+1)},{\a*\m/(\m*\n+1)})--({\a/(\m*\n+1)},{\a*\n/(\m*\n+1)});
\draw[dashed]
(-0.5,-0.25)--(3,1.5)
(-0.25,-0.5)--(2,4)
(3.5,-0.25)--(0,1.5)
(3.25,-0.5)--(1,4);

\draw

(0.9162277660168379, 1.041886116991581)--(0.758113883008419, 0.7256583509747431)--(0.441886116991581, 0.883772233983162);
 
\draw
(2.083772233983162, 1.041886116991581)--(1.925658350974743, 1.358113883008419)--(2.241886116991581, 1.516227766016838);

  \end{scope}

\end{tikzpicture}
\end{center}
\caption{These four pair of intervals all satisfy (\ref{perpprop}) but only the first forms an orthodiagonal quadrilateral} \label{perpint}
\end{figure}

The similarity of (\ref{perpprop}) to (\ref{para}) allows a reuse of the previous techniques to create a result on orthodiagonal quadrilaterals, similar to Theorem~\ref{maint}. One notable difference is that two intervals coming from two different rulings of reguli may form any of the arrangements in Figure~\ref{perpint}, instead of exclusively forming orthodiagonal quadrilaterals.

\begin{theorem}\label{orthot} Let $\mathfrak{I}$ be a set of $N$ distinct intervals in $\mathbb{R}^2$. If more than $N^{3/2}\log N$ pairs of intervals form orthodiagonal quadrilaterals, then one of the following holds. 
\begin{enumerate}[1.]
\item There are two perpendicular lines in $\mathbb{R}^2$ such that $\gtrsim N^{1/2}$ intervals have an endpoint on each line.
\item There are two perpendicular lines $\ell_1,\ell_2 \subset \mathbb{R}^2$ such that $\gtrsim N^{1/2}$ intervals $(a,b;c,d) \in \mathfrak{I}$ satisfy $(a,c) \in \ell_1$ and $(b,d) \in \ell_2$.
\item There are two subsets $\mathfrak{I}_1,\mathfrak{I}_2 \subset \mathfrak{I}$ such that for any $i_1 \in \mathfrak{I}_1$ and any $i_2 \in \mathfrak{I}_2$, the intervals $i_1,i_2$ satisfy an equation in (\ref{perpprop}). In addition, $|\mathfrak{I}_1||\mathfrak{I}_2| \gtrsim N$. 

\end{enumerate}
\end{theorem}

In order to prove this theorem, one maps the interval $(a,b,c,d)$ to the line
\begin{equation*}\label{Lperp} \mathcal{L^{\perp}}(a,b;c,d) = \left \{ \begin{pmatrix} b \\ -a \\ 0  \end{pmatrix} + t\begin{pmatrix} c \\ d \\ 1  \end{pmatrix} \colon t \in \mathbb{R} \right\} \subset \mathbb{R}^3 \end{equation*}
instead of using the map $\mathcal{L}$ above. Under this alternative correspondence, a pair of lines $\mathcal{L^{\perp}}(a,b;c,d)$ and $\mathcal{L^{\perp}}(a',b';c',d')$ intersect precisely when the left equation of (\ref{perpprop}) is satisfied. In Figure~\ref{perppull} we showcase two instances of sets of intervals resulting from pulling back rulings of reguli by $\mathcal{L}^\perp$.\\

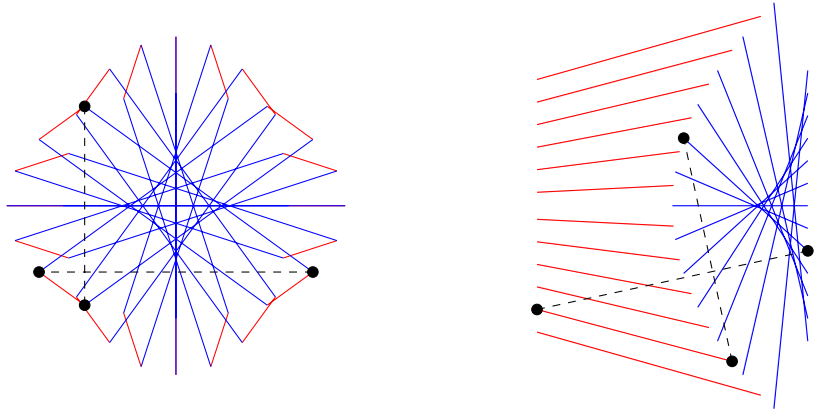
\begin{figure}[H]
\begin{center}
\begin{tikzpicture}[scale=0.3]

\draw[red]
( 5.0 , 0.0 )--( 7.5 , 0.0 )
( 4.755282581475767 , 2.3176274578121054 )--( 7.132923872213651 , 1.545084971874737 )
( 4.045084971874737 , 4.408389392193548 )--( 6.067627457812106 , 2.938926261462366 )
( 2.9389262614623664 , 6.0676274578121046 )--( 4.408389392193549 , 4.045084971874736 )
( 1.5450849718747373 , 7.132923872213651 )--( 2.317627457812106 , 4.755282581475767 )
( 3.061616997868383e-16 , 7.5 )--( 4.592425496802574e-16 , 5.0 )
( -1.5450849718747368 , 7.132923872213652 )--( -2.317627457812105 , 4.755282581475768 )
( -2.938926261462365 , 6.067627457812106 )--( -4.408389392193548 , 4.045084971874737 )
( -4.045084971874736 , 4.408389392193549 )--( -6.0676274578121046 , 2.9389262614623664 )
( -4.755282581475767 , 2.3176274578121063 )--( -7.132923872213651 , 1.5450849718747375 )
( -5.0 , 9.184850993605148e-16 )--( -7.5 , 6.123233995736766e-16 )
( -4.755282581475768 , -2.3176274578121046 )--( -7.132923872213652 , -1.5450849718747364 )
( -4.045084971874737 , -4.408389392193548 )--( -6.067627457812106 , -2.938926261462365 )
( -2.9389262614623664 , -6.0676274578121046 )--( -4.408389392193549 , -4.045084971874736 )
( -1.5450849718747377 , -7.132923872213651 )--( -2.3176274578121068 , -4.755282581475767 )
( -9.184850993605148e-16 , -7.5 )--( -1.3777276490407724e-15 , -5.0 )
( 1.5450849718747361 , -7.132923872213652 )--( 2.317627457812104 , -4.755282581475768 )
( 2.9389262614623646 , -6.067627457812107 )--( 4.408389392193547 , -4.045084971874738 )
( 4.045084971874736 , -4.408389392193549 )--( 6.0676274578121046 , -2.9389262614623664 )
( 4.755282581475767 , -2.317627457812107 )--( 7.132923872213651 , -1.5450849718747381 );
\draw[blue]
( -5.0 , 0.0 )--( 7.5 , -0.0 )
( -4.755282581475767 , 2.3176274578121054 )--( 7.132923872213651 , -1.545084971874737 )
( -4.045084971874737 , 4.408389392193548 )--( 6.067627457812106 , -2.938926261462366 )
( -2.9389262614623664 , 6.0676274578121046 )--( 4.408389392193549 , -4.045084971874736 )
( -1.5450849718747373 , 7.132923872213651 )--( 2.317627457812106 , -4.755282581475767 )
( -3.061616997868383e-16 , 7.5 )--( 4.592425496802574e-16 , -5.0 )
( 1.5450849718747368 , 7.132923872213652 )--( -2.317627457812105 , -4.755282581475768 )
( 2.938926261462365 , 6.067627457812106 )--( -4.408389392193548 , -4.045084971874737 )
( 4.045084971874736 , 4.408389392193549 )--( -6.0676274578121046 , -2.9389262614623664 )
( 4.755282581475767 , 2.3176274578121063 )--( -7.132923872213651 , -1.5450849718747375 )
( 5.0 , 9.184850993605148e-16 )--( -7.5 , -6.123233995736766e-16 )
( 4.755282581475768 , -2.3176274578121046 )--( -7.132923872213652 , 1.5450849718747364 )
( 4.045084971874737 , -4.408389392193548 )--( -6.067627457812106 , 2.938926261462365 )
( 2.9389262614623664 , -6.0676274578121046 )--( -4.408389392193549 , 4.045084971874736 )
( 1.5450849718747377 , -7.132923872213651 )--( -2.3176274578121068 , 4.755282581475767 )
( 9.184850993605148e-16 , -7.5 )--( -1.3777276490407724e-15 , 5.0 )
( -1.5450849718747361 , -7.132923872213652 )--( 2.317627457812104 , 4.755282581475768 )
( -2.9389262614623646 , -6.067627457812107 )--( 4.408389392193547 , 4.045084971874738 )
( -4.045084971874736 , -4.408389392193549 )--( 6.0676274578121046 , 2.9389262614623664 )
( -4.755282581475767 , -2.317627457812107 )--( 7.132923872213651 , 1.5450849718747381 );

\draw[dashed]
( -4.045084971874737 , 4.408389392193548 )-- ( -4.045084971874737 , -4.408389392193548 )
( 6.067627457812106 , -2.938926261462366 )--( -6.067627457812106 , -2.938926261462365 );

\filldraw
( -4.045084971874737 , 4.408389392193548 ) circle (20/3 pt) 
( -4.045084971874737 , -4.408389392193548 ) circle (20/3 pt)
( 6.067627457812106 , -2.938926261462366 ) circle (20/3 pt)
( -6.067627457812106 , -2.938926261462365 ) circle (20/3 pt);



\begin{scope}[xshift = 220mm, scale = 2]
\filldraw[blue]

( 0 , 0 )--( 3 , 0 )
( 0.0625 , 0.75 )--( 3 , -0.5 )
( 0.25 , 1.5 )--( 3 , -1.0 )
( 0.5625 , 2.25 )--( 3 , -1.5 )
( 1.0 , 3.0 )--( 3 , -2.0 )
( 1.5625 , 3.75 )--( 3 , -2.5 )
( 2.25 , 4.5 )--( 3 , -3.0 )
( 0.0625 , -0.75 )--( 3 , 0.5 )
( 0.25 , -1.5 )--( 3 , 1.0 )
( 0.5625 , -2.25 )--( 3 , 1.5 )
( 1.0 , -3.0 )--( 3 , 2.0 )
( 1.5625 , -3.75 )--( 3 , 2.5 )
( 2.25 , -4.5 )--( 3 , 3.0 );
\filldraw[red]
( 0.0225 , 0.44999999999999996 )--( -3 , 0.3 )
( 0.16000000000000003 , 1.2000000000000002 )--( -3 , 0.8 )
( 0.42250000000000004 , 1.9500000000000002 )--( -3 , 1.3 )
( 0.81 , 2.7 )--( -3 , 1.8 )
( 1.3224999999999998 , 3.4499999999999997 )--( -3 , 2.3 )
( 1.9599999999999997 , 4.199999999999999 )--( -3 , 2.8 )
( 0.0225 , -0.44999999999999996 )--( -3 , -0.3 )
( 0.16000000000000003 , -1.2000000000000002 )--( -3 , -0.8 )
( 0.42250000000000004 , -1.9500000000000002 )--( -3 , -1.3 )
( 0.81 , -2.7 )--( -3 , -1.8 )
( 1.3224999999999998 , -3.4499999999999997 )--( -3 , -2.3 )
( 1.9599999999999997 , -4.199999999999999 )--( -3 , -2.8 );

\draw[dashed]
( 1.3224999999999998 , -3.4499999999999997 )--( 0.25 , 1.5 )
( -3 , -2.3 )--( 3 , -1.0 );

\filldraw
( 1.3224999999999998 , -3.4499999999999997 ) circle (10/3 pt)
( 0.25 , 1.5 ) circle (10/3 pt)
( -3 , -2.3 ) circle (10/3 pt)
( 3 , -1.0 ) circle (10/3 pt);

\end{scope}

\end{tikzpicture}
\end{center}
\caption{Two configurations of intervals with many pairs having endpoints on two perpendicular lines.}
\label{perppull}
\end{figure}

One can also adapt the method to treat a generalisation of the trapezoids problem considered above. Two intervals form a trapezoid if two of the edges of their convex hull are parallel, but our proof did not rely in an important way on this parallel property. Indeed by mapping the interval $(a,b;c,d)$ to the line
\begin{equation*}\label{L} \mathcal{L^{\rho}}(a,b;c,d) = \left \{ \begin{pmatrix} b \\ d \\ 0  \end{pmatrix} + t\begin{pmatrix} a \\ \rho c \\ 1  \end{pmatrix} \colon t \in \mathbb{R} \right\} \subset \mathbb{R}^3,\end{equation*}
one obtains a correspondence under which a pair of lines $\mathcal{L^{\rho}}(a,b;c,d)$ and $\mathcal{L^{\rho}}(a',b';c',d')$ intersect precisely when
\begin{equation*}(a-a')(d-d')=\rho(b-b')(c-c'),\end{equation*}
i.e. the slopes formed by the endpoints of the intervals have ratio $\rho$. Thus, the arguments from Section 2 can now be applied, leading to an analogue of Theorem~\ref{maint} in the case of a fixed ratio $\rho$.

\section{Acknowledgements} The research of the first author was supported in part by a Four Year Doctoral Fellowship from the University of British Columbia. The research of the second author was supported in part by an NSERC Discovery grant and OTKA K 119528 grant. The work of the second author was also supported by the European Research Council (ERC) under the European Union's Horizon 2020 research and innovation programme (grant agreement No. 741420, 617747, 648017). The research of the third author was supported in part by Killam and NSERC doctoral scholarships.

\end{document}